\documentclass[10pt]{amsart}
\usepackage{amssymb, amsthm}\usepackage{amsmath}\usepackage{epsfig}
\usepackage{fancyhdr}

\theoremstyle{plain}
\newtheorem{Lem}{Lemma}[section]

\newtheorem{Prop}[Lem]{Proposition}
\newtheorem{Cor}[Lem]{Corollary}

\newtheorem{The}[Lem]{Theorem}

\theoremstyle{definition}

\newtheorem{Def}[Lem]{Definition}
\newtheorem{Rem}[Lem]{Remark}
\newtheorem{Exe}[Lem]{Example}

\def\Conj{\mathop{\mathcal{R}\!(\mathcal{C})}}
\def\Conjplus{\mathop{\mathcal{R^{\scriptscriptstyle +}}\!(\mathcal{C})}}
\def\ConjplusA{\mathop{\mathcal{R}^{\scriptscriptstyle +}\!(A^{\scriptscriptstyle +})}}
\def\ConjplusAnu{\mathop{\mathcal{R}^{\scriptscriptstyle+\nu}\!(A^{\scriptscriptstyle +})}}
\def\ConjA{\mathop{\mathcal{R}\!(A)}}
\def\ConjAnu{\mathop{\mathcal{R}^{\scriptscriptstyle\nu}\!(A)}}
\def\CC{\langle\mathop{\mathcal{C}}\rangle}
\def\CNC{\langle\mathop{\mathcal{G}(\mathcal{C})}\rangle}
\def\CCA{\langle\mathop{A}^{\scriptscriptstyle +}\rangle}
\def\CPhi{\langle\Phi\rangle}
\def\CDelta{\langle\Delta\rangle}
\def\Ciota{\langle\iota\rangle}

\def\NGC{\mathop{\mathcal{N}\langle\mathcal{G}(\mathcal{C})\rangle}}
\def\NA{\mathop{\mathcal{N}\langle A\rangle}}
\def\Sh{\mathop{\rm SH}}
\def\QZ{\mathop{\rm QZ}}
\def\Lsh{\mathop{\rm sh}}
\def\Rsh{\mathop{\widetilde{\rm sh}}}

\title{Parabolic subgroups of Garside groups II : ribbons}
\author{Eddy GODELLE}
\date{\today}
\begin{document}
\maketitle
\begin{abstract}We introduce and investigate the ribbon groupoid associated with a Garside group. Under a technical hypothesis, we prove that this category is a  Garside groupoid.  We decompose this groupoid into a semi-direct product of two of its parabolic subgroupoids and provide a groupoid presentation. In order to established the latter result, we describe quasi-centralizers in Garside groups. All results hold in the particular case of Artin-Tits groups of spherical type. \end{abstract}

\section{Introduction}
In 1969, Garside solved the conjugacy problem in the braid group by investigating the monoid of positive braids~\cite{Gar}. At the end of the 90's, Dehornoy and Paris introduced in~\cite{DeP} the notion of a Garside group, that captures most of the properties used in Garside approach. After this seminal paper, the notion has been extended and numerous articles have dealt with Garside monoid/group structures (\cite{Bes},\cite{ChM},\cite{Deh7},\cite{God_jal2},\cite{Pic2} for instance). Actually various groups can be equipped with a Garside group structure. Recently, Krammer~\cite{Kra}, Bessis~\cite{Bess} and Digne-Michel~\cite{DiM} have extended the notion of a Garside group into the notion of a Garside groupoid, which turned out to be crucial in the proof of the long-standing question of the $K(\pi,1)$ property for complex reflection arrangements \cite{Bess}.

The standard examples of Garside groups are braid groups and their generalization, the so-called Artin-Tits groups of spherical type. One of the main properties of Artin-Tits groups is the existence of a family of distinguished subgroups, namely the \emph{standard parabolic subgroups}, which are the subgroups generated by a subset of the standard generating set. In order to describe  normalizers of such subgroups, we were naturally led in \cite{Godthese} (see also \cite{God_jal}) to introduce a groupoid called the \emph{ribbon groupoid}. In an unpublished paper~\cite{DiM} Digne and Michel prove that the ribbon groupoid of an Artin-Tits group of spherical type is a Garside groupoid. Actually, this result implicit in~\cite{God_nonpubl}, even though the notion of a Garside groupoid had not been introduced yet. In~\cite{God_jal2}, we extended the notion of a parabolic subgroup into the framework of Garside groups. Our objective here is to go further and extend the notion of ribbon groupoid. We prove that in the context of Garside groups, ribbon groupoids remain a Garside groupoids. Indeed we prove (see next sections for definition) that  
\begin{The} \label{ThintroR(A)}For every Garside group~$A$ with a~$\nu$-structure, the positive ribbon category~$\ConjplusA$ is a Garside category and the ribbon groupoid~$\ConjA$ is the associated Garside groupoid. 
\end{The}
Every Artin-Tits group of spherical type has a $\nu$-structure. Then, Theorem~\ref{ThintroR(A)} applied in this special case. In order to prove Theorem~\ref{ThintroR(A)}, we consider the category~$\CC$ of subcategories of a category~$\mathcal{C}$, whose morphisms are natural transformations. It turns out that for a Garside category~$\mathcal{C}$, the category~$\CC$ is not a Garside category in general. However, this category shares nice properties with Garside monoids (see Section~\ref{section2section} for precise definition and results). 
We also provide a groupoid presentation of the ribbon groupoid~$\ConjA$ for every Garside group~$A$ with a~$\nu$-structure. To obtain this presentation we investigate the quasi-centralizer~$QZ(A)$ of~$A$, that is the subgroup of~$A$ made of those elements that permute the atoms of the monoid~$A^{\scriptscriptstyle +}$. In particular we generalize a result of Picantin~\cite{Pic2}:  
\begin{The}\label{ThintroR(A)2}
For every Garside group~$A$, the group~$QZ(A)$ is a finitely generated free commutative group. 
\end{The}
Furthermore, we describe the normalizer of the standard parabolic subgroups by proving (see Theorem~\ref{theodescnormal}):  
\begin{The} \label{Thintroprodsemdir}Let $A$ be a Garside group. Let~$Conj(A)$ be the  groupoid whose objects are the parabolic subgroups~$A_X$ of $A$, whose Hom-set from~$A_X$ to $A_Y$ is $\{g\in A\mid g^{-1}A_Xg = A_Y\}$, and whose composition is the product in~$A$. Let $\mathcal{P}(A)$ be the totally disconnected groupoid whose objects are the parabolic subgroups of $A$ and whose vertex group at $A_X$ is $A_X$. Then, there exists a standard parabolic subgroupoid~$\ConjAnu$ of~$\ConjA$ such that $Conj(A)$ is equal to the semi-direct product of groupoids~$\mathcal{P}(A)\rtimes \ConjAnu$. 
\end{The}
This theorem generalizes~\cite[Prop. 4.4]{God_jal}. Finally, the subgroups of the vertex groups of the ribbon groupoid~$\ConjA$ that are generated by the atoms contained in the vertex group are Garside groups. This result is similar to~\cite[Theorem~B]{BrHo}, which deals with Coxeter groups.

The reader may wonder what is the interest of proving that ribbon groupoids are Garside groupoids. The main point is that it provides a complete description of the connection between standard parabolic subgroups. The notion of a parabolic subgroup is crucial to build important tools like the so-called \emph{Deligne complexes} of Artin-Tits groups. Therefore to understand their normalizers and their mutually conjugating elements is certainly crucial. Furthermore the fact that ribbon groupoids are Garside groupoids shows that, even in order to study classical groups such as braid groups, the notion of a Garside groupoid is natural and useful. Finally, our groupoid presentations of ribbon groupoids of spherical type Artin-Tits groups provide Garside groupoids with \emph{braid-like} defining relations. These groupoids are then special among Garside groupoids and therefore should provide seminal examples of Garside groupoids that are not groups.

The paper is organized as follows. In Section~2, we introduce the notions of a Garside groupoid and of a parabolic subgroupoid; we investigate the category~$\CC$. In Section~3, we introduce the positive ribbon category associated with a Garside category and prove Theorem~\ref{ThintroR(A)}. Finally, in Section~4 we provide a groupoid presentation of the ribbon groupoid; along the way we prove Theorem~\ref{ThintroR(A)2}. We conclude with the proof of Theorem~\ref{Thintroprodsemdir}.  
\section{Garside groupoids}
\label{section2section}
Here we set the necessary terminology on \emph{small categories}. In particular we recall the notion of a \emph{natural transformation} between two \emph{functors}.  we introduce the central notion of a \emph{Garside groupoid} and provide examples.
\subsection{Small category}
The present paper deals with categories, but the reader should not be scared since we only consider \emph{small categories}. Let us introduce the notion of a small category and some basic notions in a naive way.
\label{definsmcat}

A \emph{quiver}~$\Gamma$ is a pair~$(V,E)$ of sets equipped with two maps~$s:E\cup V\to V$ and~$t:E\cup V\to V$ such that~$s(v) = t(v)  = v$ for every element~$v$ of~$V$. The elements of~$V$ are called the \emph{vertices} and the elements of~$E$ are called the edges. An edge~$e$ is said to be \emph{oriented} from its \emph{source}~$s(e)$ to its \emph{target}~$t(e)$. The \emph{path semigroup~$\mathcal{M}(\Gamma)$ associated with the quiver}~$\Gamma$  is the semigroup with~$0$ defined by the presentation whose generating set is~$V\cup E$ and whose defining relations are~$s(e)e = e = e\,t(e)$ for every $e$ in~$V\cup E$ and~$v_1v_2 = 0$ for every two \emph{distinct} elements of $V$. Clearly, we can identify the elements of~$E\cup V$ with their images in~$\mathcal{M}(\Gamma)$. Note that a product~$e_1\cdots e_k$ of edges is not equal to~$0$ in~$\mathcal{M}(\Gamma)$ if and only if for every index~$i$, one has~$t(e_i) = s(e_{i+1})$. If~$e$ lies in~$\mathcal{M}(\Gamma)$ and is not~$0$, then one can define its source and its target to be~$s(e_1)$ and~$t(e_k)$ for any of its decomposition~$e_1\cdots e_k$ as a product of elements in~$V\cup E$. The reader may note that the semigroup algebra of~$\mathcal{M}(\Gamma)$ is equal to the path algebra of the quiver~$\Gamma$ as defined in~\cite{ARS}. An element of~$\mathcal{M}(\Gamma)$ distinct from~$0$ is called a \emph{path} in the sequel.

Now, let~$\Gamma$ be a quiver and let~$\equiv$ be a congruence on~$\mathcal{M}(\Gamma)$ such that~$0$ is alone in its equivalence class~$\underline{0}$ and such that two distinct $\equiv$-equivalent paths in~$\mathcal{M}(\Gamma)$ have the same source and the same target. We denote by~$\mathcal{M}(\Gamma,\equiv)$ the quotient semigroup of~$\mathcal{M}(\Gamma)$ by the congruence~$\equiv$. By assumption on~$\equiv$, to each equivalence classes of paths~$\underline{e}$ one can associate a source and a target by setting~$s(\underline{e}) = s(e)$ and~$t(\underline{e}) = t(e)$, respectively. The (\emph{small}) \emph{category~$\mathcal{C}(\Gamma,\equiv)$ associated with the quiver~$\Gamma$ and the congruence~$\equiv$} is the set~$\mathcal{M}(\Gamma,\equiv) - \{\underline{0}\}$ of non-zero equivalence classes of paths equipped with the source and target maps, and the partial product induced by the product in~$\mathcal{M}(\Gamma,\equiv)$. In that context, the vertices of~$\Gamma$ are called the objects of the category, and the elements of the category are called the \emph{morphisms} of the category. The \emph{free category}~$\mathcal{C}(\Gamma)$ on~$\Gamma$ is obtained by considering the trivial congruence.

In the sequel, we denote by~$V(\mathcal{C})$ the set of objects of a category~$\mathcal{C}$. For every~$x$ in~$V(\mathcal{C})$, we denote by~$1_x$ the equivalence class of the trivial path~$x$. For~$x$ and~$y$ in~$V(\mathcal{C})$, we denote by~$\mathcal{C}_{x\to y}$ the set of morphisms from~$x$ to~$y$. In order to be consistent with \cite{Deh7,Bess} and with the above naive definition of a small category, if~$v$ lies in~${\mathcal{C}}_{x\to y}$ and~$w$ lies in~$\mathcal{C}_{y\to z}$, then we denote by~$vw$ the morphism of~$\mathcal{C}_{x \to z}$ obtained by composition. We denote by~$\mathcal{C}_{x\to\cdot}$ and by~$\mathcal{C}_{\cdot\to x}$ the set of morphisms of the category~$\mathcal{C}$ whose source and target, respectively, are~$x$. The seminal example to keep in mind is the case of a category with one object~$x$: this is just a monoid whose unity is~$1_x$.
\subsection{Functor and natural transformation}
Here we introduce the notion of a \emph{natural transformation} which is crucial in our article.
A map~$\Phi:\mathcal{C}\to \mathcal{C}'$ from a category~$\mathcal{C} = \mathcal{C}(\Gamma,\equiv)$ to a category~$\mathcal{C}' = \mathcal{C}(\Gamma',\equiv')$ is a \emph{functor} if, firstly, there exists a map~$\phi: V(\mathcal{C})\to V(\mathcal{C})$ such that~$\Phi(1_x) = 1_{\phi(x)}$ for every vertex~$x$ in~$V(\mathcal{C})$ and, secondly, one extends~$\Phi$ to a morphism of semigroups from~$\mathcal{M}(\Gamma,\equiv)$ to~$\mathcal{M}(\Gamma',\equiv')$ by setting~$\Phi(\underline{0}) = \{\underline{0}\}$. Note that the second property is equivalent to say that~$\Phi(\mathcal{C}_{x\to y})$ is included in~$\mathcal{C}'_{\phi(x)\to\phi(y)}$ for every two objects~$x,y$ of~$\mathcal{C}$, and~$\Phi(vw) = \Phi(v)\phi(w)$ for every two morphisms~$w$,~$v$ of~$\mathcal{C}$ such that~$t(v) = s(w)$. For instance, for every category~$\mathcal{C}(\Gamma,\equiv)$, there is a surjective functor from the free category~$\mathcal{C}(\Gamma)$ to~$\mathcal{C}(\Gamma,\equiv)$. An \emph{isomorphism of categories} is a functor that is bijective. In that case, the inverse is also an isomorphism of categories and the set~$\Phi(\mathcal{C}_{x\to y})$ is equal to~$\mathcal{C}_{\Phi(x)\to \Phi(y)}$. An automorphism of~$\mathcal{C}$ is an isomorphism from~$\mathcal{C}$ to itself. 
 
Let~$\mathcal{C} = \mathcal{C}(\Gamma,\equiv)$ and~$\mathcal{C}' = \mathcal{C}(\Gamma',\equiv')$ be two categories with~$\Gamma = (V,E)$. Let~$\Phi$ and~$\Psi$ be two functors from~$\mathcal{C}$ to~$\mathcal{C'}$. Then, a \emph{natural transformation}~$\Delta$ from~$\Phi$ to~$\Psi$ is  a map~$\Delta: V\to\mathcal{C}'$ that sends every vertex~$x$ in~$V$ to a morphism~$\Delta(x)$ from~$\Phi(x)$ to~$\Psi(x)$ such that for every two objects~$x,y$ in~$\Gamma$ and every morphism~$v$ in~$\mathcal{C}_{x\to y}$, one has $\Delta(x)\Psi(v) = \Phi(v)\Delta(y)$ (see Figure~\ref{diagrammetransfnaturelle}). 
\begin{figure}[ht]
\begin{picture}(100,70)
\put(18,0){\includegraphics[scale = 0.4]{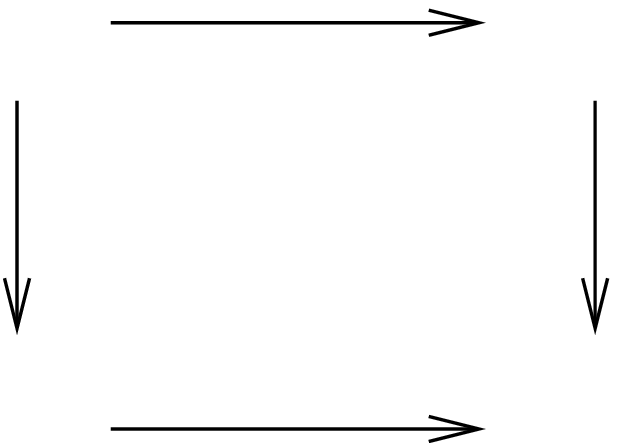}}
\put(8,0){$\Phi(y)$}\put(8,46){$\Phi(x)$}\put(78,0){$\Psi(y)$} \put(78,46){$\Psi(x)$} 
\put(0,25){$\Phi(v)$}\put(38,54){$\Delta(x)$}\put(38,6){$\Delta(y)$} \put(88,25){$\Psi(v)$} 
   
\end{picture}
\caption{Natural transformation from~$\Phi$ to~$\Psi$.}\label{diagrammetransfnaturelle}
\end{figure}

For instance, if~$\mathcal{C}$ and~$\mathcal{C}'$ are groups and~$\Phi$,~$\Psi$ are two morphisms of groups from~$\mathcal{C}$ to~$\mathcal{C}'$, then a natural transformation from~$\Phi$ to~$\Psi$ is an element~$\delta$ of the group~$\mathcal{C}'$ such that~$\Psi(h) = \delta^{-1}\Phi(h)\delta$ for every~$h$ in~$\mathcal{C}$.  

\subsection{Garside groupoid}
\label{sectiongardgrd}To introduce the central notion of a Garside groupoid, we need first to fix some definitions.

 The \emph{atoms} of a category~$\mathcal{C}(\Gamma,\equiv)$ are the non-identity morphisms whose representing paths are edges only. The \emph{atom graph}~$\mathcal{A}(\Gamma,\equiv)$ of the category is the graph that has the same set of vertices than~$\Gamma$ and whose edges are the atoms of the category. We recall that a binary relation (defined on a set) is noetherian when every decreasing infinite sequence stabilizes. We say that a morphism~$w$ of~$\mathcal{C}(\Gamma,\equiv)$ is a factor of the morphism~$w'$ when there exist two morphisms~$w_1$ and~$w_2$ such that~$w' = w_1w'w_2$. We say that the category~$\mathcal{C}(\Gamma,\equiv)$ is \emph{noetherian} if the set of morphisms is noetherian for factors, in other words, if for every infinite sequence~$(w_i)_{i\in\mathbb{N}}$ of morphisms such that~$w_{i+1}$ is a factor of~$w_i$, there exists an index~$N$ such that~$w_i = w_N$ for every~$i\geq N$.  In this case, the category is \emph{atomic}: each morphism can be written as a finite product of atoms; furthermore,the category~$\mathcal{C}(\Gamma,\equiv)$ is isomorphic to~$\mathcal{C}(\mathcal{A}(\Gamma,\equiv),\equiv')$, where~$\equiv'$ is the congruence on~$\mathcal{C}(\mathcal{A}(\Gamma,\equiv))$ induced by the congruence~$\equiv$. This property holds in particular when for every morphism there is a bound on the lengths of its representing paths in the quiver~$\Gamma$ (where each edge has length~$1$). If in the category~$\mathcal{C}(\Gamma,\equiv)$ there is no loop for the factor relation, then for every object~$x$ the left-divisibility and the right-divisibility induce partial orders on the sets~$\mathcal{C}_{x\to\cdot}$ and~$\mathcal{C}_{\cdot\to x}$, respectively. This is the case when the category is noetherian. We say that the category~$\mathcal{C}(\Gamma,\equiv)$ is cancellative when for every two objects~$x$,~$y$ and every morphism~$v$ in~$\mathcal{C}_{x\to y}$ the maps~$w\mapsto vw$ and~$w\mapsto wv$ are injective on the sets~$\mathcal{C}_{y\to\cdot}$ and~$\mathcal{C}_{\cdot\to x}$, respectively.   

Finally, a \emph{groupoid} is a category~$\mathcal{C}$ such that every morphism has an inverse, in other words, such that for every morphism~$v$ in~$\mathcal{C}_{x\to y}$ there exists a morphism~$\tilde{v}$ in~$\mathcal{C}_{y\to x}$ with~$v\tilde{v} = 1_x$ and~$\tilde{v}v = 1_y$. To each category~$\mathcal{C} = \mathcal{C}(\Gamma,\equiv)$ one can associate a groupoid~$\mathcal{G}(\mathcal{C}) = \mathcal{C}(\tilde{\Gamma},\tilde{\equiv})$ and a functor~$\iota: \mathcal{C}\to \mathcal{G}(\mathcal{C})$ in the following way. The quiver~$\tilde{\Gamma}$ is obtained by adding to~$\Gamma$ an edge~$\tilde{v}$ from~$t(v)$ to~$s(v)$ for each edge~$v$ of~$\Gamma$. The congruence~$\tilde{\equiv}$ is the congruence generated by~$\equiv$ and the relations~$v\tilde{v} = 1_x$ and~$\tilde{v}v = 1_y$ for each edge~$v$ from~$x$ to~$y$ in~$\Gamma$. We call the groupoid~$\mathcal{G}(\mathcal{C})$ the \emph{groupoid of formal inverses} of~$\mathcal{C}$. For instance, the \emph{free groupoid} on~$\Gamma$ is the groupoid~$\mathcal{G}(\mathcal{C}(\Gamma))$. It should be noted that every functor~$F:\mathcal{C}\to\mathcal{D}$ between two categories induces a functor~$\mathcal{G}(F):\mathcal{G}(\mathcal{C})\to\mathcal{G}(\mathcal{D})$ such that the diagram $$\begin{array}{lcl}\ \ \ \mathcal{C}&\stackrel{F}{\hookrightarrow}&\ \ \ \mathcal{D}\\\iota_\mathcal{C}\downarrow&&\iota_\mathcal{D}\downarrow\\\ \mathcal{G}(\mathcal{C})&\stackrel{\mathcal{G}(F)}{\hookrightarrow}&\ \mathcal{G}(\mathcal{D}) \end{array}$$ is commutative.
We are now ready to introduce the definition of the special kind of groupoids we consider in this paper. We refer to~\cite{KoM} for the general theory on lattices. In \cite{Bess}, Bessis defines the notion of a Garside category as follows.
\begin{Def}[Garside category] A \emph{categorical Garside structure} is a pair~$(\mathcal{C},\Delta)$ such that:
\begin{itemize}
\item$\mathcal{C}$ is a small category equipped with an automorphism~$\Phi$ such that~$\Delta$ is a natural transformation from the identity functor~$Id: \mathcal{C}\to \mathcal{C}$ to~$\Phi$;
\item~$\mathcal{C}$ is noetherian and cancellative;
\item for every object~$x$ of~$\mathcal{C}$, the sets~$\mathcal{C}_{x \to\cdot}$ and~$\mathcal{C}_{\cdot\to x}$  are lattices for the left-divisibility order and the right-divisibility order, respectively;
\item for every two objects~$x,y$ of~$\mathcal{C}$ and every atom~$v$ in~$\mathcal{C}_{x\to y}$, there exists a morphism~$\overline{v}$ such that the following diagram is commutative:
\begin{figure}[ht]
\begin{picture}(100,70)
\put(8,0){\includegraphics[scale = 0.4]{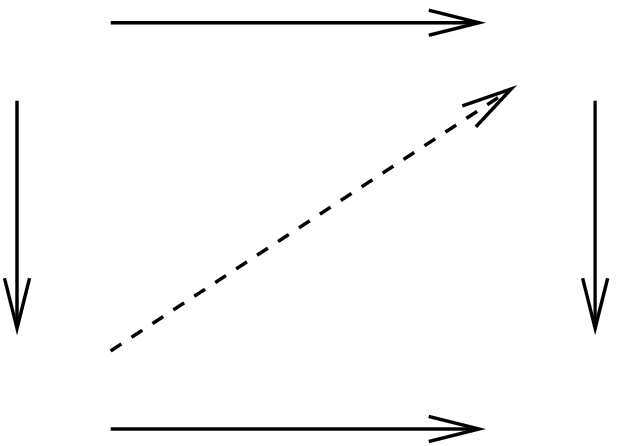}}
\put(8,0){$y$}\put(8,46){$x$}\put(68,0){$\Phi(y)$} \put(68,46){$\Phi(x)$} 
\put(2,25){$v$}\put(32,54){$\Delta(x)$}\put(32,6){$\Delta(y)$} \put(78,25){$\Phi(v)$} 
\put(37,28){$\overline{v}$}
\end{picture}
\end{figure}
\end{itemize} 
In this case, we say that the category~$\mathcal{C}$ is a \emph{Garside category}.\\
We say that~$(\mathcal{C},\Delta)$ is a \emph{monoidal Garside structure} when furthermore the category~$\mathcal{C}$ has a unique object. In this case, we say that the monoid~$\mathcal{C}$ is a \emph{Garside monoid}.
\end{Def}
The reader should note that in the above definition, the automorphism~$\Phi$ is uniquely defined by $\Delta$.
\begin{Def}[Garside groupoid] A \emph{Garside groupoid}~$\mathcal{G}(\mathcal{C})$ is the groupoid of formal inverses of a Garside category~$\mathcal{C}$. When furthermore~$\mathcal{C}$ is a monoid, then we say that~$\mathcal{G}(\mathcal{C})$ is a Garside group. 
\end{Def}
If~$\mathcal{C}$ and~$\mathcal{C}'$ are two categories and~$\Phi$ is a functor, then the subset~$\Phi(\mathcal{C}')$ of~$\mathcal{C}$ has a natural structure of category. We say that such a subset is a \emph{subcategory} of~$\mathcal{C}$. In the sequel, we denote by $\iota_{\mathcal{P}}$ the \emph{canonical functor} from the subcategory~$\mathcal{P} = \Phi(\mathcal{C}')$ to the category. The reader should note that, by hypothesis, a Garside category~$\mathcal{C}$ verifies the so-called \emph{Ore relations} and, therefore, the functor~$\iota: \mathcal{C}\to \mathcal{G}(\mathcal{C})$ is injective. In the sequel we consider~$\mathcal{C}$ as a subcategory of~$\mathcal{G}(\mathcal{C})$. Clearly, the automorphism~$\Phi$ extends to an automorphism of~$\mathcal{G}(\mathcal{C})$, still denoted by~$\Phi$, and that $\Delta$ extends to a natural transformation from the identity functor~$Id:\mathcal{G}(\mathcal{C})\to \mathcal{G}(\mathcal{C})$ to~$\Phi$.  

\begin{Exe}\label{ExagpAT}The seminal case of a Garside groupoid~$\mathcal{G}(\mathcal{C})$ occurs when the category has one object~$x$. In this case, the category~$\mathcal{C}$ is a Garside monoid and~$\mathcal{G}(\mathcal{C})$ is its group of fractions, that is a Garside group (\cite{Deh7,Bes}). The element~$\Delta(x)$ is called the \emph{Garside element} of the group~$\mathcal{G}(\mathcal{C})$, or of the monoid~$\mathcal{C}$. In such a Garside group, a natural transformation from the identity morphism to an automorphism corresponds to an element that realizes this automorphism as an inner automorphism. \emph{Braid groups} and, more generally, \emph{Artin-Tits groups of spherical type} are Garside groups which are defined by a group presentation of the type \begin{equation}\label{presarttitsgrps}
\langle S|\underbrace{sts\ldots}_{m_{s,t}\ terms} = \underbrace{tst\ldots}_{m_{s,t}\ terms}~;\ \forall s,t\in S, s\not= t\ \rangle \end{equation} where~$m_{s,t}$ are positive integers greater than 2 with~$m_{s,t} = m_{t,s}$. In this case, the associated \emph{Coxeter group}, which obtained by adding the relations~$s^2 = 1$ for~$s$ in~$S$ to the presentation, is finite. If we consider~$S = \{s_1,\ldots, s_n\}$ with~$m_{s_i,s_j}= 3$ for~$|i-j| =1$ and $m_{s_i,s_j} = 2$ otherwise, we obtain the classical presentation of the braid group~$B_{n+1}$ on~$n+1$ strings. In that case, the associated Coxeter group is the permutation group on~$n+1$ elements. If~$A$ is such an Artin-Tits group, and~$W$ is its Coxeter group, then the canonical surjection~$\iota:A\to W$ has a canonical section \cite[Chap.~4]{Bou}. If we denote by~$\Delta$ the image by that section of the greatest element of~$W$, then~$\Delta$ belongs to the submonoid~$A^{\scriptscriptstyle +}$ of~$A$ generated by~$S$, and the associated inner group automorphism induces an automorphism of~$A^{\scriptscriptstyle +}$. It turns out that the pair~$(A^{\scriptscriptstyle +},\Delta)$ is a monoidal Garside structure with~$A$ for associated Garside group. A group can be a Garside group for several monoidal Garside structures. For instance, the braid group~$B_{n+1}$ has the alternative group presentation:
$$\left\langle a_{ts} \left| \begin{array}{lcl} a_{ts}a_{rq} = a_{rq}a_{ts}&;& (t-r)(t-q)(s-r)(s-q) > 0 \\ a_{ts}a_{sr} = a_{sr}a_{tr} = a_{tr}a_{ts}&;&t>s>r\end{array}\right.\right\rangle$$ where~$n+1\geq t\geq s\geq 1$ for the generators~$a_{ts}$,  and if~$B^{BKL+}_{n+1}$ denotes the submonoid of~$B_{n+1}$ generated by the~$a_{ts}$, and~$\delta = a_{(n+1)n}\cdots a_{21}$, then the inner group automorphism of~$B_{n+1}$ induced by~$\delta$ restricts to an automorphism of~$B^{BKL+}_{n+1}$. Again, the pair~$(B^{BKL+}_{n+1},\delta)$ is a monoidal Garside structure with~$B_{n+1}$ for associated Garside group. This alternative Garside structure is called the \emph{dual Garside structure} of the braid group. This point of view was generalized by Bessis in \cite{Bes} to every Artin-Tits groups of spherical type.
\end{Exe}
\begin{Exe}\cite{Bes,McC} The free group~$F_2$ on two letters~$a,b$ has a presentation $$\langle a_i,\ i\in\mathbb{Z} \mid a_ia_{i+1} = a_ja_{j+1},\ i,j\in\mathbb{Z} \rangle$$ where~$a_0 = a$ and~$a_1 = b$. Let~$\Delta = a_0a_1$. If we denote by~$F_2^{\scriptscriptstyle +}$ the submonoid of~$F_2$ generated by the elements~$a_i$, and by~$\Phi$ the automorphism of~$F_2$ that sends~$a_i$ onto~$a_{i+2}$, then~$(F_2^{\scriptscriptstyle +},\Delta)$ is a monoidal Garside structure  with group~$F_2$ as the associated Garside group. The automorphism induced by~$\Delta$ is~$\Phi$ .\label{exef2} 
\end{Exe}
We end this subsection with a definition and two properties of Garside groupoids that we need in the sequel. They are immediate generalizations of well-known properties of Artin-Tits groups.
\begin{Def}[Simple element]\cite[Def. 2.2]{Bess} \label{defisimplelem} Let~$(\mathcal{C},\Delta)$ be a categorical Garside structure and denote by~$\Phi$ the automorphism of~$\mathcal{C}$ induced by~$\Delta$. A morphism~$v$ in~$\mathcal{C}_{x\to y}$ is \emph{simple} if there exists~$v'$ in~$\mathcal{C}_{y\to\Phi(x)}$ such that~$vv' = \Delta(x)$.
 \end{Def}
Note that in the above definition the morphism~$v'$ is also simple. By definition of a Garside category, the atoms are simple.
\begin{Prop}\label{uniquedecomposprop} Let~$(\mathcal{C},\Delta)$ be a categorical Garside structure.\\ (i) \cite{Bess} Every morphism~$v$ that belongs to~$\mathcal{C}_{x\to y}$ can be decomposed as a product~$s_1\cdots s_k$ of simple elements. \\(ii) Let~$v$ lie in~$\mathcal{G}(\mathcal{C})_{x\to y}$. There exist an object~$z$, a morphism~$v_1$ in~$\mathcal{C}_{x\to z}$ and a morphism~$v_2$ in~$\mathcal{C}_{y\to z}$ such that~$v_1$ and~$v_2$ are prime to each other in~$\mathcal{C}_{\cdot\to z}$ and~$v$ is equal to~$v_1v_2^{-1}$. The triple~$(z,v_1,v_2)$ is unique. 
\end{Prop}
Point~$(ii)$ of the above proposition does not appear in~\cite{Bess}, but the proof is similar to the case of Garside groups. We call the triple~$(z,v_1,v_2)$ the \emph{right greedy normal form} of~$v$ There is a similar unique \emph{left greedy normal form}~$v_3^{-1}v_4$.
\subsection{The category of subcategories}

\begin{Def}[Category of subcategories] \label{defcatsbcat}Let~$\mathcal{C}$  be a (small) category. We define \emph{the category of subcategories}~$\CC$ of~$\mathcal{C}$ as follows: 
\begin{itemize}
 \item the objects are the non-empty subcategories of~$\mathcal{C}$;
\item  if~$\mathcal{P}$ and~$\mathcal{Q}$ are two subcategories of~$\mathcal{C}$, then~$\delta$ lies in~$\CC_{\mathcal{P}\to\mathcal{Q}}$ if there exists an isomorphism~$\phi$ from~$\mathcal{P}$ to~$\mathcal{Q}$ such that~$\delta$ is a natural transformation from~$\iota_\mathcal{P}$ to~$\iota_\mathcal{Q}\circ\phi$.
\item If~$\delta_1$ and~$\delta_2$ belong to~$\CC_{\mathcal{P}\to\mathcal{Q}}$ and~$\CC_{\mathcal{Q}\to\cdot}$, respectively, then  the composition morphism~$\delta_1\delta_2$ is defined by~$\delta_1\delta_2(x) = \delta_1(x)\delta_2(\phi_1(x))$, where~$\phi_1$ is the isomorphism associated with $\delta_1$.
\end{itemize}
\end{Def}
Note that in the second item of Definition~\ref{defcatsbcat} the isomorphism~$\phi$ is uniquely defined by~$\delta$ because~$\mathcal{C}$ is cancellative; in the third item,~$\delta_1\delta_2$ is a natural transformation from~$\iota_\mathcal{P}$ to~$\iota_{\mathcal{Q}'}\circ\phi_2\circ\phi_1$, where~$\mathcal{Q}'$ is such that $\delta_2$ belongs to $\CC_{\mathcal{Q}\to\cdot\mathcal{Q}'}$. 

\begin{figure}[ht]
\begin{picture}(150,70)
\put(8,0){\includegraphics[scale = 0.4]{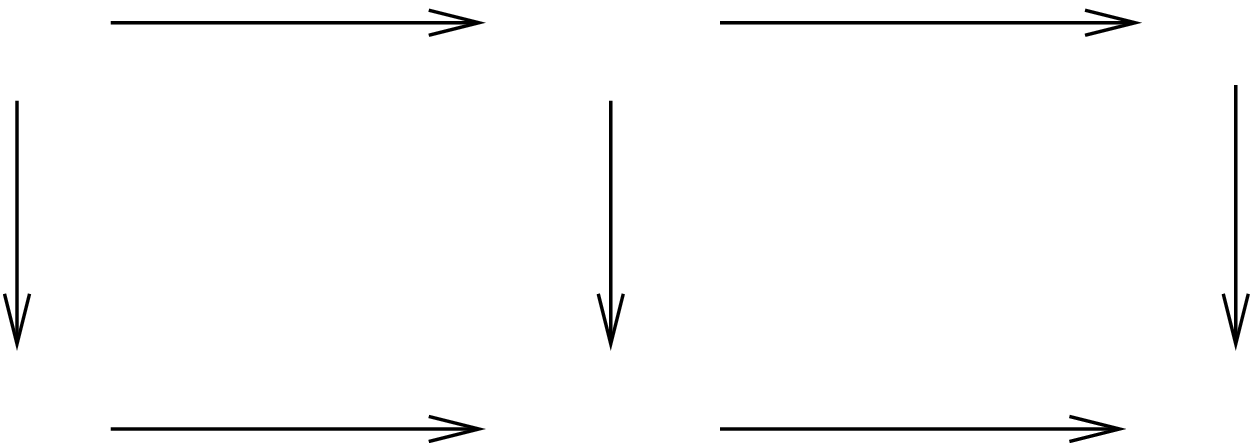}}
\put(8,0){$y$}\put(8,46){$x$}\put(68,0){$\phi_1(y)$} \put(68,46){$\phi_1(x)$} \put(142,0){$\phi_2(\phi_1(y))$} \put(142,46){$\phi_2(\phi_1(x))$} 
\put(2,25){$v$}\put(32,54){$\delta_1(x)$}\put(32,6){$\delta_1(y)$} \put(80,25){$\phi_1(v)$} 
\put(98,54){$\delta_2(\delta_1(x))$}\put(98,6){$\delta_2(\delta_1(y))$} \put(155,25){$\phi_2(\phi_1(v))$}
\end{picture}
\end{figure}
Our objective below is to investigate the properties of~$\CC$ and, in particular, to decide whether or not~$\CC$ is a Garside category.
\begin{Prop}\label{lempropCCcat}Let~$(\mathcal{C},\Delta)$ be a categorical Garside structure and denote by~$\Phi$ the automorphism of~$\mathcal{C}$ induced by~$\Delta$.\\ (i) There exist an automorphism~$\CPhi$ of~$\CC$ and a natural transformation~$\CDelta$ from the identity functor of~$\CC$ to~$\CPhi$ that are induced by $\Phi$ and $\Delta$.\\(ii) The category~$\CC$ is cancellative.\\(iii) For every object~$\mathcal{P}$ of~$\CC$, the left-divisibility and the right-divisibility define partial orders on~$\CC_{\mathcal{P}\to\cdot}$ and~$\CC_{\cdot\to\mathcal{P}}$, respectively.
\end{Prop}
\begin{proof}
{\it(i)} Let~$\mathcal{P}$,~$\mathcal{Q}$ be two objects of~$\CC$ and assume~$\delta$ lies in~$\CC_{\mathcal{P}\to\mathcal{Q}}$; denote by~$\phi:\mathcal{P}\to\mathcal{Q}$ the isomorphism defined by~$\delta$. Then,~$\Phi\circ\phi\circ\Phi^{-1}$ is an isomorphism from~$\Phi(\mathcal{P})$ to~$\Phi(\mathcal{Q})$, and the map~$\Phi\circ\delta\circ\Phi^{-1}$ is a natural transformation from the canonical embedding~$\iota_{\Phi(\mathcal{P})}$ to~$\iota_{\Phi(\mathcal{Q})}\circ \Phi\circ\phi\circ\Phi^{-1}$. Therefore, we define an isomorphism~$\CPhi$ of~$\CC$ by setting~$\CPhi(\mathcal{P}) = \Phi(\mathcal{P})$ and~$\CPhi(\delta) = \Phi\circ\delta\circ\Phi^{-1}$. Now, for every object~$\mathcal{P}$ of~$\CC$ we denote by~$\CDelta(\mathcal{P})$ the restriction of the map~$\Delta$ to the subcategory~$\mathcal{P}$. Then,~$\CDelta(\mathcal{P})$ lies in~$\CC_{\mathcal{P}\to\cdot}$. Indeed, the map~$\CDelta(\mathcal{P})$ is a natural transformation from~$\iota_{\mathcal{P}}$ to~$\iota_{\Phi(\mathcal{P})}\circ\Phi_{\mathcal{P}}$, where~$\Phi_{\mathcal{P}}$ is the restriction of~$\Phi$ to~$\mathcal{P}$. Clearly, the map~$\CDelta$ is a natural transformation from the identity functor of~$\CC$ to~$\CPhi$.\\
{\it (ii)} The cancellativity of the category~$\CC$ follows directly from the cancellativity of the category~$\mathcal{C}$.\\ 
{\it (iii)} There is no loop for the factor relation in~$\CC$ because there is no loop for the factor relation in~$\mathcal{C}$. Therefore,~$\CC_{\mathcal{P}\to\cdot}$ and~$\CC_{\cdot\to\mathcal{P}}$ are partially ordered sets for the left-divisibility and the right-divisibility, respectively. 
\end{proof}

In the sequel, an  object~$x$ of a category~$\mathcal{C}$  such that $\mathcal{C}_{x\to\cdot} = \mathcal{C}_{\cdot\to x} = \{1_x\}$ is called a \emph{trivial object} of the category.

\begin{Prop}Let~$(\mathcal{C},\Delta)$ be a categorical Garside structure such that $\mathcal{C}$ has no trivial object.\label{proppropribbcat} The category~$\CC$ is noetherian if and only if its object set is finite.\end{Prop}
\begin{proof} 
Assume~$\mathcal{C}$ has a finite number of objects. Let~$(\delta_n)_{n\in \mathbb{N}}$ be a sequence of morphisms of~$\CC$ such that for every~$n$, the morphism~$\delta_{n+1}$ is a factor of the morphism~$\delta_n$. We denote by~$\phi_n:\mathcal{P}_n\to\mathcal{Q}_n$ the isomorphism defined by~$\delta_n$. Let~$(\theta_n)_{n\geq 1}$ and~$(\eta_n)_{n\geq 1}$ be the sequences of morphisms of~$\CC$ such that~$\delta_n = \theta_{n+1}\delta_{n+1}\eta_{n+1}$, and denote by~$(\psi_n)_{n\geq 1}$ and~$(\rho_n)_{n\geq 1}$ the isomorphisms associated with~$(\theta_n)_{n\geq 1}$ and~$(\eta_n)_{n\geq 1}$, respectively. Then, for every positive integer~$n$, every object~$x$ of~$\mathcal{P}_n$ and every morphism~$v$ of~$\mathcal{P}_n$, we have~$\phi_n(v) = \eta_{n+1}(\phi_{n+1}(\psi_{n+1}(v)))$ and~$\delta_n (x) = \theta_{n+1}(x)\delta_{n+1}(\psi_{n+1}(x))\eta_{n+1}(\phi_{n+1}(\psi_{n+1}(x)))$. Since~$\mathcal{C}$ is noetherian, for every object~$x$ of~$\mathcal{P}_0$, there exists some integer~$N_x$ such that for every~$n\geq N_x$, we have~$\delta_{n}(\psi_n\circ\cdots\circ\psi_1(x)) = \delta_{N_i}(\psi_{N_i}\circ\cdots\circ\psi_1(x))$. In particular, for every~$n\geq N$ we have $\psi_{n+1}(\psi_n\circ\cdots\circ\psi_1(x_i)) = \psi_n\circ\cdots\circ\psi_1(x_i)$, 
$\theta_{n+1}(\psi_n\circ\cdots\circ\psi_1(x_i)) = 1_{\psi_n\circ\cdots\circ\psi_1(x_i)}$, $\rho_{n+1}(\phi_n\circ\psi_n\circ\cdots\circ\psi_1(x_i)) = \phi_n\circ\psi_n\circ\cdots\circ\psi_1(x_i)$ and $\eta_{n+1}(\phi_n\circ\psi_n\circ\cdots\circ\psi_1(x_i)) = 1_{\phi_n\circ\psi_n\circ\cdots\circ\psi_1(x_i)}$. 
Let $N = \max\{N_x\mid x\in V(\mathcal{P}_0)\}$.  From the above discussion, it follows that for every~$n\geq N$ one has $\mathcal{P}_i = \mathcal{P}_N$, $\mathcal{Q}_i  = \mathcal{Q}_N$ and $\delta_i = \delta_N$.\\
Assume finally that~$\mathcal{C}$ has an infinite number of objects; denote by~$(x_i)_{i\geq 0}$ a sequence of distinct objects of~$\mathcal{C}$ such that~$\mathcal{C}_{x_{2i}\to x_{2i+1}}$ is not empty for every~$i$. We fix a morphism~$v_i$ in~$\mathcal{C}_{x_{2i}\to x_{2i+1}}$ for each~$i$. Let~$\mathcal{Q}_i$ be the subcategory of~$\mathcal{C}$ whose set of objects is~$\{x_{2j}\mid j\geq i\}\cup\{x_{2j+1}\mid j< i\}$ such that~$\mathcal{Q}_i$ has trivial objects only. Let $\mathcal{Q}_\infty$ be the subcategory whose object set is~$\{x_{2j+1}\mid j\geq 0\}$  such that~$\mathcal{Q}_\infty$ has trivial objects only. Consider~$\delta_i$ the map from the set of objects of~$\mathcal{Q}_i$ to the morphisms~$\CC$ that sends~$x_{2j}$ on~$v_j$ if~$j\geq i$ and~$x_{2j+1}$ on~$1_{x_{2j+1}}$ otherwise. Let~$\tau_i$ be the map from the objects of~$\mathcal{Q}_i$ to the morphisms of~$\CC$ that sends~$x_{2i}$ on~$v_i$ and~$x_k$ on~$1_{x_k}$ otherwise. Then,~$\delta_i$ clearly belongs to~$\CC_{\mathcal{Q}_i\to\mathcal{Q}_\infty}$ and~$\tau_i$ belongs to~$\CC_{\mathcal{Q}_i\to\mathcal{Q}_{i+1}}$. By construction,~$\delta_{i} = \delta_{i+1}\tau_i$ for every~$i$. Hense, the category~$\CC$ is not noetherian.\end{proof}

\begin{Lem}Consider a categorical Garside structure~$(\mathcal{C},\Delta)$ and denote by~$\Phi$ the automorphism of~$\mathcal{C}$ induced by~$\Delta$. Assume~~$\mathcal{P}$ is an object of~$\mathcal{C}$ such  that $\CC_{\mathcal{P}\to\cdot}$ is a lattice for the left-divisibility. Assume~$\delta$ is an atom of the category~$\CC$ that belongs to~$\CC_{\mathcal{P}\to\mathcal{Q}}$. Then, there exists a morphism~$\overline{\delta}$ in~$\CC_{\CPhi(\mathcal{P})\to\mathcal{Q}}$ such that the following diagram is commutative:
\begin{figure}[ht]
\begin{picture}(100,70)
\put(8,0){\includegraphics[scale = 0.42]{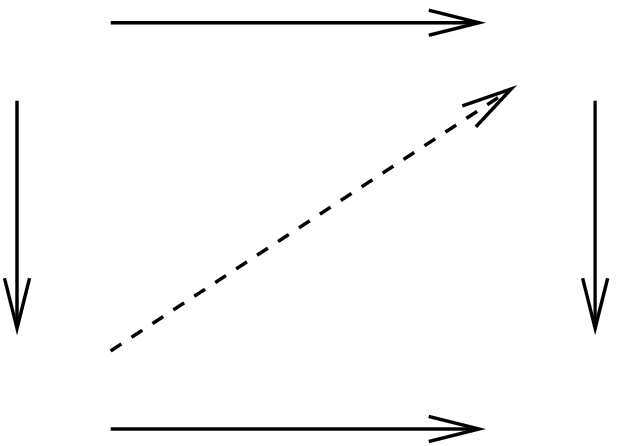}}
\put(6,0){$\mathcal{Q}$}\put(6,48){$\mathcal{P}$}\put(70,0){$[\Phi](\mathcal{Q})$} \put(70,48){$[\Phi](\mathcal{P})$} 
\put(2,25){$\delta$}\put(30,55){$[\Delta](\mathcal{P})$}\put(30,6){$[\Delta](\mathcal{Q})$} \put(82,25){$[\Phi](\delta)$} 
\put(40,27){$\overline{\delta}$}
\end{picture}
\end{figure}
\end{Lem}
\begin{proof}
 Indeed, since~$\delta$ is an atom, the greatest common divisor of~$\delta$ and~$\CDelta(\mathcal{P})$ is either trivial or equal to~$\delta$. But,~$\delta$ is a non-identity morphism of~$\CC$. Therefore, there exists an object~$x$ of~$\mathcal{P}$ such that~$\delta(x)$ is not (the identity morphism)~$x$. This implies that~$\delta(x)\!\wedge\!\Delta(x)$ is not trivial and the claim follows. 
\end{proof}
Gathering the above results we get: 
\begin{Cor} \label{Thgsfvs2} Consider a categorical Garside structure~$(\mathcal{C}\Delta)$ whose object set is finite. Assume that for every object~$\mathcal{P}$ of $\CC$ the sets~$\CC_{\mathcal{P}\to\cdot}$ and~$\CC_{\cdot\to\mathcal{P}}$ are lattices for the left-divisibility and the right-divisibility, respectively. Then, the pair~$(\CC,\CDelta)$ is a categorical Garside structure.\end{Cor}
\begin{proof}

 The category~$\CC$ is cancellative and noetherian, and for every non empty subcategory~$\mathcal{P}$ of~$\mathcal{C}$ the sets~$\CC_{\mathcal{P}\to\cdot}$ and~$\CC_{\cdot\to\mathcal{P}}$ are lattices for the left-divisibility and the right-divisibility, respectively.  Furthermore,~$\CDelta$ is a natural transformation from the identity functor of~$\CC$ to the automorphism~$\CPhi$, where~$\Phi$ is the automorphism of~$\mathcal{C}$ induced by~$\Delta$. 
\end{proof}

There is no clear argument to expect that the left/right-divisibilities induce lattice orders in general. However there is one particular case where this property always hold: 
\begin{Prop}\label{Thgsfvs3} Consider a monoidal Garside structure~$(\mathcal{C},\Delta)$. Then, for every non-empty subcategory~$\mathcal{P}$ of~$\mathcal{C}$ the sets~$\CC_{\mathcal{P}\to\cdot}$ and~$\CC_{\cdot\to\mathcal{P}}$ are lattices for the left-divisibility and the right-divisibility.
\end{Prop}

\begin{proof}
Consider a categorical Garside structure~$(\mathcal{C},\Phi,\Delta)$ and a subcategory~$\mathcal{P}$ of~$\mathcal{C}$. Let~$\delta_1$,~$\delta_2$ belong to~$\CC_{\mathcal{P}\to\cdot}$. For every object~$x$ of~$\mathcal{P}$, we set~$\delta_1\!\vee\!\delta_2(x) = \delta_1(x)\!\vee\!\delta_2(x)$ and~$\delta_1\!\wedge\!\delta_2(x) = \delta_1(x)\!\wedge\!\delta_2(x)$ (we recall that~$\mathcal{C}_{x\to\cdot}$ is a lattice). Denote by~$\phi_1$ and~$\phi_2$ the isomorphisms associated with $\delta_1$ and $\delta_2$, respectively. We denote by~$\phi_1\!\vee\!\phi_2(x)$ and~$\phi_1\!\wedge\!\phi_2(x)$ the objects of~$\mathcal{C}$ such that~$\delta_1\!\vee\!\delta_2(x)$ and~$\delta_1\!\wedge\!\delta_2(x)$ lie in~$\mathcal{C}_{\cdot\to \phi_1\!\vee\!\phi_2(x)}$ and~$\mathcal{C}_{\cdot\to \phi_1\!\wedge\!\phi_2(x)}$, respectively. Let~$y$ be another object of~$\mathcal{P}$, and assume there exists~$v$ in~$\mathcal{P}_{x\to y}$. Write~$\delta_1\!\vee\!\delta_2(x) = \delta_1(x)\theta_1(x) = \delta_2(x)\theta_2(x)$ and~$\delta_1\!\vee\!\delta_2(y) = \delta_1(y)\theta_1(y) = \delta_2(y)\theta_2(y)$. We have the sequence of equalities~$\delta_2(x)\phi_2(v)\theta_2(y) = v\delta_2(y)\theta_2(y) = v\delta_1(y)\theta_1(y) = \delta_1(x)\phi_1(v)\theta_1(y)$. Since~$\mathcal{C}_{x\to\cdot}$ is a lattice, there exists a unique morphism~$\phi_1\!\vee\!\phi_2(v)$ in~$\mathcal{C}_{\phi_1\!\vee\!\phi_2(x)\to\phi_1\!\vee\!\phi_2(y)}$ which verifies the equalities~$\phi_1(v)\theta_1(y)=\theta_1(x)(\phi_1\!\vee\!\phi_2)(v)$ and~$\phi_2(v)\theta_2(y)=\theta_2(x)(\phi_1\!\vee\!\phi_2)(v)$. Consider the map~$\phi_1\!\vee\!\phi_2$ that sends an object~$x$ of~$\mathcal{P}$ to~$\phi_1\!\vee\!\phi_2(x)$ and an edge~$v$ of~$\mathcal{P}$ to~$\phi_1\!\vee\!\phi_2(v)$. Then, by the cancellativity property, for every two objects~$x,y$ of~$\mathcal{P}$, the restriction of this map to~$\mathcal{P}_{x\to y}$ is injective. In particular, when~$\mathcal{P}$ has one object,~$\phi_1\!\vee\!\phi_2$ is an isomorphism from the subcategory~$\mathcal{P}$ to the subcategory~$\phi_1\!\vee\!\phi_2(\mathcal{P})$. Assume that $\mathcal{C}$ is a monoid, that is a category with one object. Set~$\alpha_i = (\phi_1\!\vee\!\phi_2)\circ \phi_i^{-1}$ for~$i = 1,2$. Clearly, the maps~$x\mapsto \theta_1(x)$,~$x\mapsto\theta_2(x)$ and~$x\mapsto \delta_1\!\vee\!\delta_2(x)$ are natural transformations from~$\iota_{\phi_1(\mathcal{P})}$,~$\iota_{\phi_2(\mathcal{P})}$ and~$\iota_{\mathcal{P}}$ to~$\iota_{\phi_1\!\vee\!\phi_2(\mathcal{P})}\circ \alpha_1$,~$\iota_{\phi_1\!\vee\!\phi_2(\mathcal{P})}\circ \alpha_2$ and~$\iota_{\phi_1\!\vee\!\phi_2(\mathcal{P})}\circ (\phi_1\!\vee\!\phi_2)$, respectively. Thus,~$\delta_1\!\vee\!\delta_2$ is a common multiple of~$\delta_1$ and~$\delta_2$ in~$\CC$. Now, assume~$\tau$ is a common multiple of~$\delta_1$ and~$\delta_2$ in~$\CC$ with~$\tau = \delta_1\tau_1 = \delta_1\tau_2$. Let $\psi$ the isomorphism associated with~$\tau$. Then, for every object~$x$ of~$\mathcal{P}$, there exists~$\rho(x)$ in~$\mathcal{C}_{\phi_1\!\vee\!\phi_2(x)\to\tau(x)}$ such that~$\tau_1(x) = \theta_1(x)\rho(x)$ and~$\tau_2(x) = \theta_2(x)\rho(x)$. Therefore, the map~$\rho:x\mapsto\rho(x)$ lies in~$\CC_{(\phi_1\!\vee\!\phi_2)(\mathcal{P})\to\psi(\mathcal{P})}$, and~$\delta_1\!\vee\!\delta_2$ is the least common multiple of~$\delta_1$ and~$\delta_2$ in~$\CC_{\mathcal{P}\to\cdot}$. Similarly, for every morphism~$v$ in~$\mathcal{P}_{x\to y}$ there exists a unique morphism~$\phi_1\!\wedge\!\phi_2(v)$ in~$\mathcal{C}_{\phi_1\!\wedge\!\phi_2(x)\to \phi_1\!\wedge\!\phi_2(y)}$ such that~$v\ \delta_1\!\wedge\!\delta_2\!(y) = \delta_1\!\wedge\!\delta_2\!(x)\ \phi_1\!\wedge\!\phi_2\!(v)$. Again, if $\mathcal{C}$ is a monoid the map~$\phi_1\!\wedge\!\phi_2\!$ is injective, and~$\delta_1\!\wedge\!\delta_2$ is the greatest common divisor of~$\delta_1$ and~$\delta_2$ in~$\CC_{\mathcal{P}\to\cdot}$. Thus, the set~$\CC_{\mathcal{P}\to\cdot}$ is a lattice for the left-divisibility. By similar arguments, the set~$\CC_{\cdot\to\mathcal{P}}$ is a lattice for the right-divisibility. 
\end{proof}

As a direct consequence of Corollary~\ref{Thgsfvs2} and Proposition~\ref{Thgsfvs3}, we get
\begin{The} Let~$(\mathcal{C},\Delta)$ be a monoidal Garside structure. Then,~$(\CC,\CDelta)$ is a categorical Garside structure.\label{Thgsfvs4}
\end{The}
\begin{proof} 
\end{proof}
\subsection{The groupoid of subgroupoids}
A subcategory that is a groupoid will be called a \emph{subgroupoid} in the sequel. When~$(\mathcal{C},\Delta)$ is a categorical Garside structure such that $\CC$ is a Garside category, then~$\mathcal{G}(\CC)$ is a Garside groupoid. There is another way to associate a groupoid with~$\mathcal{C}$. Indeed, 
\begin{Def}[Groupoid of subgroupoids] Let~$\mathcal{G}$ be a small groupoid. The \emph{groupoid of subgroupoids} of~$\mathcal{G}$ is the full subcategory~$\mathcal{N}\langle\mathcal{G}\rangle$ of the category~$\langle\mathcal{G}\rangle$ whose objects are the subgroupoids of~$\mathcal{G}$. \label{defngc}
\end{Def} 
Clearly, the category~$\mathcal{N}\langle\mathcal{G}\rangle$ is a groupoid. Then,~$\mathcal{N}\langle\mathcal{G}(\mathcal{C})\rangle$ is a groupoid for every small category~$\mathcal{C}$. If~$G$ is a group and~$P$ is a subgroup, then~$\mathcal{N}\langle G\rangle_{P\to P}$ is the normalizer of~$P$ in~$G$. 
 
We recall that we denote by~$\mathcal{G}(F): \mathcal{G}(\mathcal{C})\to \mathcal{G}(\mathcal{D})$ the functor induced by a functor~$F:\mathcal{C}\to\mathcal{D}$.
\begin{Prop}
 Let~$\mathcal{C}$ be a category. The canonical functor~$\iota:\mathcal{C}\to\mathcal{G}(\mathcal{C})$ induces a functor~$\Ciota: \CC\to\NGC$. 
\end{Prop}
\begin{proof} For every object~$\mathcal{P}$ of the category~$\CC$ we set $\langle \mathcal{P} \rangle = \mathcal{G}(\iota_{\mathcal{P}})(\mathcal{G}(\mathcal{P}))$. In other words, $\langle \mathcal{P} \rangle$ is the subgroupoid of~$\mathcal{G}(\mathcal{C})$  generated by the subcategory~$\iota(\mathcal{P})$.  Now, consider~$\delta$ in~$\CC_{\mathcal{P}\to\mathcal{Q}}$, and denote by~$\phi$ its associated isomorphism. Let~$\langle\delta\rangle$ be the map that sends an object~$x$ of~$\langle \mathcal{P} \rangle$ to the morphism~$\iota(\delta(x))$. Then~$\langle\delta\rangle$ belongs to~$\NGC_{\langle \mathcal{P} \rangle\to\langle \mathcal{Q} \rangle}$.  Indeed, Consider the map~$\langle\phi\rangle : \langle \mathcal{P} \rangle \to \langle \mathcal{Q} \rangle$ which sends an object~$x$ of~$\langle \mathcal{P} \rangle$ to~$\phi(x)$ and a morphism~$v$ in~$\langle \mathcal{P} \rangle_{x\to y}$ to~$(\langle\delta \rangle(x))^{-1}\,v\,\langle\delta\rangle(y)$. Then~$\langle\phi\rangle$  is an isomorphism from $\langle \mathcal{P} \rangle$ to $\langle \mathcal{Q} \rangle$ and $\langle\delta\rangle$ is a natural transformation from~$\iota_{\langle \mathcal{P} \rangle}$ to $\iota_{\langle \mathcal{Q} \rangle}\circ \langle\phi\rangle$. Then we obtain a map~$\Ciota: \CC\to\NGC$ by setting $\Ciota(\mathcal{P}) = \langle \mathcal{P} \rangle$ and $\Ciota(\delta) =\langle\delta\rangle$. Clearly, this map is a functor. 
\end{proof}
We remark that~$\Ciota$ is not injective in general: for two distinct subcategories~$\mathcal{P}$ and~$\mathcal{Q}$ we can have~$\Ciota(\mathcal{P}) = \Ciota(\mathcal{Q})$. For instance, consider the braid monoid~$B^{\scriptscriptstyle +}_3$ which is defined by the monoid presentation~$\langle s,t \mid sts = tst\rangle$; its group of fractions is the braid group~$B_3$. The sets~$\{s,t\}$ and~$\{s, sts\}$ generate distinct submonoids of~$B_3^{\scriptscriptstyle +}$, but both generate the braid group~$B_3$, as a group. Moreover, the functor~$\Ciota$ is neither faithful in general: the restriction of~$\Ciota$ to a set~$\CC_{\mathcal{P}\to\mathcal{Q}}$ is not always injective; if we consider the monoid~$M$ defined by the monoid presentation~$\langle a,b,c\mid a^2 = a\ ;\ ba = ac\rangle$, and we denote by~$M_b$,~$M_c$ the submonoids generated by~$b$ and~$c$, respectively, then~$a$ is a non-trivial morphism in~$\langle M\rangle_{M_b \to M_c}$, whereas~$\iota(a)$ is trivial in~$\mathcal{G}(M)$.
 
In the sequel, we write $\mathcal{G}\CC$ for $\mathcal{G}(\CC)$.
\begin{Prop}\label{propfaithfulpropert} Let~$(\mathcal{C},\Delta)$ be a monoidal Garside structure. Denote by~$\iota$ the canonical embedding functor of~$\mathcal{C}$ into~$\mathcal{G}(\mathcal{C})$. The functor~$\mathcal{G}\!\Ciota: \mathcal{G}\CC\to \NGC$ is faithful. \end{Prop}
\begin{proof} Let~$\mathcal{P}$ and~$\mathcal{Q}$ be two objects of~$\mathcal{G}\CC$ and consider~$\delta$,~$\tau$ in~$\mathcal{G}\CC_{\mathcal{P}\to\mathcal{Q}}$ such that~$\mathcal{G}\!\Ciota(\delta) = \mathcal{G}\!\Ciota(\tau)$. Then,~$\mathcal{G}\!\Ciota(\tau^{-1}\delta) = 1_{\mathcal{G}(\mathcal{P})}$. Then, it is enough to assume for the rest of the proof that~$\mathcal{P} = \mathcal{Q}~$ and~$\tau = 1_{\mathcal{P}}$. Since the category~$\CC$ is a Garside category, we can identify~$\CC$  with its image in its group of fractions~$\mathcal{G}\CC$, and decompose the morphism~$\delta$ as a product~$\delta_1\delta_2^{-1}$  with~$\delta_1$ and~$\delta_2$ in~$\CC_{\mathcal{P}\to \mathcal{P}'}$ for some object~$\mathcal{P}'$ of~$\CC$. Since~$\mathcal{G}\!\Ciota(\delta) = 1_{\mathcal{G}(\mathcal{P})}$, we have~$\mathcal{G}\!\Ciota(\delta_1)= \mathcal{G}\!\Ciota(\delta_2)$, that is,~$\Ciota(\delta_1)= \Ciota(\delta_2)$. By definition, this means that for every object~$x$ of~$\mathcal{P}$, we have~$\iota(\delta_1(x)) = \iota(\delta_2(x))$. But~$\iota$ is an injective functor. Then~$\delta_1 = \delta_2$ and~$\delta = 1_{\mathcal{P}}$. 
\end{proof}
As a final comment to this section, we remark that the morphism~$\Ciota$ is not injective on the objects. Therefore, the groupoid~$\mathcal{G}\CC$ cannot be seen as a subgroupoid of~$\mathcal{N}\langle\mathcal{G}(\mathcal{C})\rangle$. 
\section{The ribbon groupoid}
In this section, we introduce and investigate the \emph{positive ribbon category} and the \emph{ribbon groupoid}. This ribbon category is defined as a full subcategory of the category~$\CC$. The ribbon groupoid is defined as a full subgroupoid of~$\mathcal{N}\langle\mathcal{G}(\mathcal{C})\rangle$, which turns out to be, under some restriction, the group of fractions of the positive ribbon category. We first define the notion of a parabolic subgroupoid.
\subsection{Parabolic subgroupoids}
Here we extend the notion of a parabolic subgroup of a Garside group~\cite{God_jal2} into the framework of Garside groupoids. If~$\mathcal{C}$ is a category and~$\delta$ lies in~$\mathcal{C}_{x\to\cdot}$, then we denote by~$[1_x,\delta]_{\mathcal{C}}$ the set of elements between~$1_x$ and~$\delta$ for the left-divisibility in~$\mathcal{C}_{x\to\cdot}$. 
\begin{Def}[Parabolic subcategories]\label{defparasucat}
Let~$(\mathcal{C},\Delta)$ be a categorical Garside structure. A subcategory~$\mathcal{P}$ of~$\mathcal{C}$ is a \emph{parabolic subcategory} of~$\mathcal{C}$ if for every object~$x$ of~$\mathcal{P}$ the sets~$\mathcal{P}_{x\to\cdot}$ and~$\mathcal{P}_{\cdot\to x}$ are sublattices of~$\mathcal{C}_{x\to\cdot}$ and~$\mathcal{C}_{\cdot\to x}$, respectively, and, there exists~$\delta$ in~$\CC_{\mathcal{P}\to\mathcal{P}}$ such that for every object~$x$ of~$\mathcal{P}$, one has $$[1_x,\delta(x)]_\mathcal{C} = [1_x,\Delta(x)]_\mathcal{C}\cap \mathcal{P}_{x\to\cdot} = \Delta(x)\wedge\mathcal{P}_{x\to\cdot}$$ in the lattice~$\mathcal{C}_{x\to\cdot}$.
\end{Def}
With the notation of the above definition, when~$\mathcal{P}$ is a parabolic subcategory then one has~$[1_x,\Delta(x)]_\mathcal{C}\cap \mathcal{P}_{x\to\cdot} = [1_x,\delta(x)]_{\mathcal{P}}$, and the atoms of~$\mathcal{P}$ are the atoms of $\mathcal{C}$ that lie in~$\mathcal{P}$. It is immediate that the pair~$(\mathcal{P},\delta)$ is a categorical Garside structure. Furthermore, the canonical embedding functor from~$\mathcal{P}$ to~$\mathcal{C}$ induces an injective functor from~$\mathcal{G}(\mathcal{P})$ to~$\mathcal{G}(\mathcal{C})$ such that the following diagram is commutative:
$$\begin{array}{ccc}\mathcal{P}&\hookrightarrow&\mathcal{C}\\\downarrow&&\downarrow\\\mathcal{G}(\mathcal{P})&\hookrightarrow&\mathcal{G}(\mathcal{C}) \end{array}$$
Therefore, we can identify~$\mathcal{G}(\mathcal{P})$ with a subgroupoid of~$\mathcal{G}(\mathcal{C})$. In that context, we have~$\mathcal{P} = \mathcal{G}(\mathcal{P}) \cap \mathcal{C}$ (where the notion of an intersection of two subcategories is defined in an obvious way).
\begin{Def}[Parabolic subgroupoids]
Let~$(\mathcal{C},\Delta)$ be a categorical Garside structure.\\ A subgroupoid of the Garside groupoid~$\mathcal{G}(\mathcal{C})$ is a \emph{standard parabolic subgroupoid} when it is the groupoid of formal inverses~$\mathcal{G}(\mathcal{P})$ of a parabolic subcategory~$\mathcal{P}$ of~$\mathcal{C}$. A non empty subgroupoid~$\mathcal{G}$ of~$\mathcal{G}(\mathcal{C})$ is a \emph{parabolic subgroupoid} if there exists a standard parabolic subgroupoid~$\mathcal{G}(\mathcal{P})$ such that~$\NGC_{\mathcal{G}(\mathcal{P})\to \mathcal{G}}$ is not empty.     
\end{Def}
\begin{Exe}\label{exemATdeux} In the context of Garside groups, we recover the notion of a standard parabolic subgroup introduced in~\cite{God_jal2}, which extends the notion of a standard parabolic subgroup of an Artin-Tits group. Standard parabolic subgroups of an Artin-Tits group~$A$ are the subgroups generated by any set of atoms of the associated Artin-Tits monoid~$A^{\scriptscriptstyle +}$ ({\it cf.} Example~\ref{ExagpAT})). In the sequel, we denote by~$A_I$ the standard parabolic subgroup of a Garside group~$A$ generated by a set~$I$ of atoms of~$A^{\scriptscriptstyle +}$. 
\end{Exe}
We only consider parabolic subgroupoids that are standard in the sequel. Therefore, we write \emph{parabolic} for \emph{standard parabolic} in the sequel. 

\begin{Exe}\label{exef22} Consider the  free group~$F_2$ on two letters~$a,b$ and the Garside structure of Example~\ref{exef2}. In this context, Except the whole group~$F_2$ and the trivial subgroup, a parabolic subgroup of~$F_2$ is any subgroup generated by one generator~$a_i$. Then, parabolic proper subgroups are all isomorphic to~$\mathbb{Z}$.
\end{Exe}
\begin{Exe}\label{exef23} Direct product of Garside groups are Garside groups. For instance, the group~$A$ defined by the group presentation~$\langle a,b,c\mid a^2 = b^2, ac = ca, bc = cb\rangle$ is a Garside group, with~$a^2c$ as Garside element. Its parabolic subgroups are the four subgroups~$\{1\}$,~$A_{\{a,b\}}$,~$A_{\{c\}}$ and~$A$.
\end{Exe}
\begin{Exe}\label{exef24} The group~$A$ defined by the presentation~$\langle a,b,c\mid aba = bab = c^2\rangle$ is a Garside group, with~$aba$ as Garside element. Indeed, amalgam products of Garside groups above their Garside elements are Garside groups. The parabolic subgroups of~$A$ are the four subgroups~$\emptyset$,~$A_{\{a\}}$,~$A_{\{b\}}$ and~$A$.
\end{Exe}
It is quite easy to see that most of the properties proved in~\cite{God_jal2} can be extended into the context of Garside groupoid, with similar proof.
\subsection{Positive ribbon category}
We are now ready to define the ribbon groupoid, whose objects are the parabolic subgroupoids.

\begin{Def}[Positive ribbon category]
 Let~$(\mathcal{C}\Delta)$ be a categorical Garside structure. The \emph{positive ribbon category}~$\Conjplus$ associated with~$\mathcal{C}$ is the full subcategory of~$\CC$ whose objects are the parabolic subcategories of~$\mathcal{C}$. 
\end{Def}
We recall that for every parabolic subcategory~$\mathcal{P}$ of~$\mathcal{C}$, we have identified~$\mathcal{G}(\mathcal{P})$ with a subgroupoid of~$\mathcal{G}(\mathcal{C})$. We also recall that~$\mathcal{P} = \mathcal{C}\cap \mathcal{G}(\mathcal{P})$ and that the categories~$\mathcal{P}$ and~$\mathcal{G}(\mathcal{P})$ have the same objects. In particular every element~$\delta$ of $\CNC_{\mathcal{G}(\mathcal{P})\to \mathcal{G}(\mathcal{Q})}$ belongs to~$\CNC_{\mathcal{P}\to \phi(\mathcal{P})}$, where~$\phi$ is the isomorphism associated with~$\delta$.
\begin{Def}[Ribbon groupoid]
 Let~$(\mathcal{C},\Delta)$ be a categorical Garside structure.\\
(i) If~$\mathcal{G}(\mathcal{P})$ and~$\mathcal{G}(\mathcal{Q})$ are two parabolic subgroupoids, then a ribbon from the former to the latter is a morphism~$\delta$ in~$\CNC_{\mathcal{G}(\mathcal{P})\to \mathcal{G}(\mathcal{Q})}$  that belongs to~$\CNC_{\mathcal{P}\to \mathcal{Q}}$.\\ 
(ii) The \emph{ribbon groupoid}~$\Conj$ of~$\mathcal{G}(\mathcal{C})$ is the subcategory of~$\NGC$ whose objects are the standard parabolic subgroupoids of~$\mathcal{G}(\mathcal{C})$ and whose morphisms are the ribbons.
\end{Def}
\begin{Exe}Consider the  free group~$F_2$ on two letters~$a,b$, and the Garside structure of Examples~\ref{exef2} and \ref{exef22}. Then, each object of the groupoid~$\mathcal{R}(F_2)$ distinct from the trivial subgroup and the whole group can be identified with one of the~$a_i$. Clearly, there are morphisms between~$a_i$ and~$a_j$ in~$\mathcal{R}(F_2)$ if and only if~$i-j$ is even.
\end{Exe}
The category~$\Conj$ is clearly a groupoid and the image by~$\Ciota$ of a morphism of~$\Conjplus$ is a ribbon. In the sequel such a ribbon is called a \emph{positive} ribbon. Indeed,
\begin{Prop} Let~$(\mathcal{C},\Delta)$ be a categorical Garside structure such that the vertex set of $\mathcal{C}$ is finite. Denote by~$\Phi$ the automorphism of~$\mathcal{C}$ induced by~$\Delta$.\\
(i) The category~$\Conjplus$ is cancellative and noetherian.\\ 
(ii) The functor~$\CPhi$ restricts to an automorphism~$\CPhi$ of~$\Conjplus$. The restriction of~$\CDelta$ to the vertex set of~$\Conjplus$ is a natural transformation from the identity functor of~$\Conjplus$ to this restricted automorphism.\\
(iii) When $\mathcal{C}$ is a monoid, the functor~$\Ciota: \CC\to\NGC$ is injective on to the category~$\Conjplus$ and the image~$\Ciota(\Conjplus)$ is a subcategory of~$\Conj$. Moreover the morphism~$\mathcal{G}\!\Ciota$ induces an isomorphism of groupoids from~$\mathcal{G}(\Conjplus)$ to~$\Conj$.
\end{Prop}
\begin{proof} $(i)$~$\Conjplus$ is cancellative and noetherian since it is a subcategory of~$\CC$.\\$(ii)$ From the equality~$\Phi(\Delta(x)) = \Delta(\Phi(x))$ and the definition of a parabolic subcategory, it follows that~$\Phi(\mathcal{P})$ is a parabolic subcategory if and only if~$\mathcal{P}$ is a parabolic subcategory. Moreover, in this case the map~$\CDelta(\mathcal{P})$ lies in~$\Conjplus_{\mathcal{P}\to\Phi(\mathcal{P})}$.\\
$(iii)$ By construction, the image of~$\Conjplus$ by~$\Ciota$ is a subcategory of~$\Conj$. Furthermore, we already know that~$\Ciota$ is faithful by Proposition~\ref{propfaithfulpropert}. Now, if~$\mathcal{P}_1$,~$\mathcal{P}_2$ are two objects of~$\Conjplus$, then one has $$\mathcal{P}_1 = \mathcal{C}\cap \mathcal{G}(\mathcal{P}_1) = \mathcal{C}\cap \Ciota(\mathcal{P}_1) = \mathcal{C}\cap \Ciota(\mathcal{P}_2) = \mathcal{C}\cap \mathcal{G}(\mathcal{P}_2) = \mathcal{P}_2$$ Hence,~$\Ciota$ lies injective on~$\Conjplus$. 
Using that $\CC$ is a Garside category, we get that~$\mathcal{G}\!\Ciota$ induces an isomorphism of groupoids from the subgroupoid~ $\mathcal{G}(\Conjplus)$ to its image, which lies in~$\Conj$. Let~$\delta$ lie in~$\Conj_{\mathcal{G}(\mathcal{P})\to \mathcal{G}(\mathcal{Q})}$ and~$\phi:\mathcal{G}(\mathcal{P})\to \mathcal{G}(\mathcal{Q})$ be its associated isomorphism. We can decompose~$\delta$ as a product~$\delta = \delta_1 \Delta^{-n}$ for some non negative integer~$n$  and some~$\delta_1$ in the monoid~$\mathcal{C}$. Now, the element~$\delta_1$ belongs to~$\Conjplus_{\mathcal{P}\to \Phi^{n}(\mathcal{Q})}$ since~$(\Phi^n\circ\phi)(\mathcal{P}) = \Phi^n(Q)$. Therefore, the image of~$\mathcal{G}(\Conjplus)$ by~$\Ciota$ is~$\Conj$.        
\end{proof}
As a consequence of the previous result, we get
\begin{Prop} \label{trtrtrrt}Let~$(\mathcal{C},\Delta)$ be a categorical Garside structure such that~$\mathcal{C}$ is a monoid. Assume the sets~$\Conjplus_{\mathcal{P}\to\cdot}$ and~$\Conjplus_{\cdot\to\mathcal{P}}$ are sublattices of~$\CC_{\mathcal{P}\to\cdot}$ and~$\CC_{\cdot\to\mathcal{P}}$, respectively, for every parabolic submonoid~$\mathcal{P}$ of~$\mathcal{C}$.  Then~$(\Conjplus,\CDelta)$ is a categorical Garside structure with~$\Conj$ as associated Garside groupoid.
\end{Prop}
\begin{proof}
\end{proof}
Generalizing the terminology of~\cite{God_jal2} to Garside groupoids, we could say under the assumption of Proposition~\ref{trtrtrrt} that~$\Conjplus$ is a \emph{Garside subcategory} of~$\CC$.
\subsection{Garside groups with a $\nu$-structure}
In the previous sections, we prove, under some technical assumptions, that the category of subcategories of a Garside~$\CC$ category~$\mathcal{C}$ is itself a Garside category. In particular this is always the case when~$\mathcal{C}$ is a Garside monoid. We prove that, moreover, the positive ribbon category~$\Conjplus$ is also a Garside category, with the ribbon groupoid~$\Conj$ as associated Garside groupoid, under the assumption that this positive ribbon category~$\Conjplus$ respects the lattice structure of the category of subcategories~$\CC$. In this section, we consider a simple criteria which insures that~$\Conjplus$ respects the lattice structure of~$\CC$. The criteria is verified by Artin-Tits groups of spherical type.

We recall that a monoidal Garside structure~$(A^{\scriptscriptstyle +},\Delta)$ is a categorical Garside structure such that~$A^{\scriptscriptstyle +}$ is a monoid; we denote by~$A$ the associated Garside group. We denote by~$A^{\scriptscriptstyle +}_X$ a parabolic submonoid of~$A^{\scriptscriptstyle +}$ whose atom set is~$X$. We denote by~$A_X$ the associated standard parabolic subgroup. We recall that, in this context, an element~$g$ of~$\ConjplusA_{A_X^{\scriptscriptstyle +}\to A_Y^{\scriptscriptstyle +}}$ is an element of~$A^{\scriptscriptstyle +}$ such that~$g^{-1}A^{\scriptscriptstyle +}_X g = A^{\scriptscriptstyle +}_Y$ (which is equivalent to~$g^{-1}X g = Y$). 

From now on, for $g,h$ in~$(A^{\scriptscriptstyle +}$, we denote by~$g\lor h$ and~$g\land h$ their lcm and their gcd for the left-divisibility, respectively, and we denote by~$g\tilde{\lor} h$ and~$g\tilde{\land} h$ their lcm and  their gcd for the right-divisibility, respectively.
\begin{Def}[$\nu$-structure] \label{defnustructure}Let~$(A^{\scriptscriptstyle +},\Delta)$ be a monoidal Garside structure. Let~$S$ be the atom set of~$A^{\scriptscriptstyle +}$.\\
(i)  We say that~$A^+$ \emph{has a~$\nu$-function} if for each parabolic submonoid~$A^{\scriptscriptstyle +}_X$ there exists a map~$\nu_X: S\to \mathcal{R^{\scriptscriptstyle +}}(A^{\scriptscriptstyle +})_{A_X^{\scriptscriptstyle +}\to\cdot}$ such that
\begin{enumerate}
\item for every~$s$ in~$S$, the morphism~$\nu_X(s)$ is an atom of~$\ConjplusA$;
\item for every~$s$ in~$S$ and every~$g$ in~$\ConjplusA_{A_X^{\scriptscriptstyle +}\to\cdot}$, if~$s$ left-divides~$g$ then~$\nu_X(s)$ left-divides~$g$;
\item for every~$s,t$ in~$S$, the lcm of~$\nu_X(s)$ and~$\nu_X(t)$ for left-divisibility lies in~$\ConjplusA_{A_X^{\scriptscriptstyle +}\to\cdot}$. 
\end{enumerate}
(ii) We say that~$A^+$ \emph{has a~$\tilde{\nu}$-function} if for each parabolic submonoid~$A^{\scriptscriptstyle +}_X$ there exists a map~$\tilde{\nu}_X: S\to \mathcal{R^{\scriptscriptstyle +}}(A^{\scriptscriptstyle +})_{\cdot\to A_X^{\scriptscriptstyle +}}$ such that
\begin{enumerate}
\item for every~$s$ in~$S$, the morphism~$\tilde{\nu}_X(s)$ is an atom of~$\ConjplusA$;
\item  for every~$s$ in~$S$ and every~$g$ in~$\ConjplusA_{\cdot\to A_X^{\scriptscriptstyle +}}$, if~$s$ right-divides~$g$ then~$\tilde{\nu}_X(s)$ right-divides~$g$;
\item  for every~$s,t$ in~$S$, the lcm of~$\tilde{\nu}_X(s)$ and~$\tilde{\nu}_X(t)$ for right-divisibility lies in~$\ConjplusA_{\cdot\to A_X^{\scriptscriptstyle +}}$. 
\end{enumerate}
(iii) We say that~$A^+$ has a~\emph{$\nu$-structure} if it has both a~$\nu$-function and a~$\tilde{\nu}$-function.
\end{Def}

\begin{Rem} (i) We do not require that~$s$ left-divides~$\nu_X(s)$.\\(ii) When~$A$ has a~$\nu$-function ({\it resp.} a~$\tilde{\nu}$-function), then each atom of~$\mathcal{R^{\scriptscriptstyle +}}(A^{\scriptscriptstyle +})$ that belongs to~$\mathcal{R^{\scriptscriptstyle +}}(A^{\scriptscriptstyle +})_{A_X^{\scriptscriptstyle +}\to\cdot}$  ({\it resp.}  to $\mathcal{R^{\scriptscriptstyle +}}(A^{\scriptscriptstyle +})_{\cdot\to A_X^{\scriptscriptstyle +}}$) is equal to some~$\nu_X(s)$ ({\it resp.} some~$\tilde{\nu}_X(s)$). Furthermore, if~$g$ is an atom of~$\mathcal{R^{\scriptscriptstyle +}}(A^{\scriptscriptstyle +})$ that belongs to~$\mathcal{R^{\scriptscriptstyle +}}(A^{\scriptscriptstyle +})_{A_X^{\scriptscriptstyle +}\to\cdot}$  ({\it resp.}  to $\mathcal{R^{\scriptscriptstyle +}}(A^{\scriptscriptstyle +})_{\cdot\to A_X^{\scriptscriptstyle +}}$) and~$s$ in~$S$ left-divides ({\it resp.} right-divides)~$g$, then~$g = \nu_X(s)$ ({\it resp.}~$g = \tilde{\nu}_X(s)$). 
\end{Rem}

\begin{Exe} \label{exemplnufunctspherATGP}Every Artin-Tits monoid of spherical type ({\it cf.} Example~\ref{exef23}) has a~$\nu$-structure. Consider the notation of Example~\ref{exef23}. Let~$X$ be a subset of~$S$ and~$s$ lie in~$S$. If~$s$ does not belong to~$X$, we set~$\nu_X(s) = \Delta_X^{-1}\Delta_{X\cup\{s\}}$ and~$\tilde{\nu}_X(s) = \Delta_{X\cup\{s\}}\Delta_X^{-1}$; if~$s$ lies in~$X$, we set~$\nu_X(s) = \tilde{\nu}_X(s) = \Delta_{X_i}$ where~$X_i$ is the indecomposable component of~$X$ that contains~$s$. We recall that an Artin-Tits group of spherical type is said to be \emph{indecomposable} when it is not the direct product of two of its proper standard parabolic subgroups. It follows from~\cite{God_jal} that these functions provide a~$\nu$-structure to the Artin-Tits monoid.
\end{Exe}

\begin{Prop} \label{rtgvfresc} Let~$(A^{\scriptscriptstyle +},\Delta)$ be a monoidal Garside structure.\\
(i) If~$A$ has a~$\nu$-function, then the set~$\mathcal{R^{\scriptscriptstyle +}}(A^{\scriptscriptstyle +})_{A_X^{\scriptscriptstyle +}\to\cdot}$ is a sublattice of~$\CCA_{A_X^{\scriptscriptstyle +}\to\cdot}$ for every parabolic submonoid~$A^{\scriptscriptstyle +}_X$.\\
(ii) If~$A$ has a~$\tilde{\nu}$-function then the set~$\mathcal{R^{\scriptscriptstyle +}}(A^{\scriptscriptstyle +})_{\cdot\to A_X^{\scriptscriptstyle +}}$ is a sublattice of~$\CCA_{\cdot\to A_X^{\scriptscriptstyle +}}$ for every parabolic submonoid~$A^{\scriptscriptstyle +}_X$.
\end{Prop}
\begin{figure}[ht]
\begin{picture}(150,110)
\put(10,15){\includegraphics[scale = 0.4]{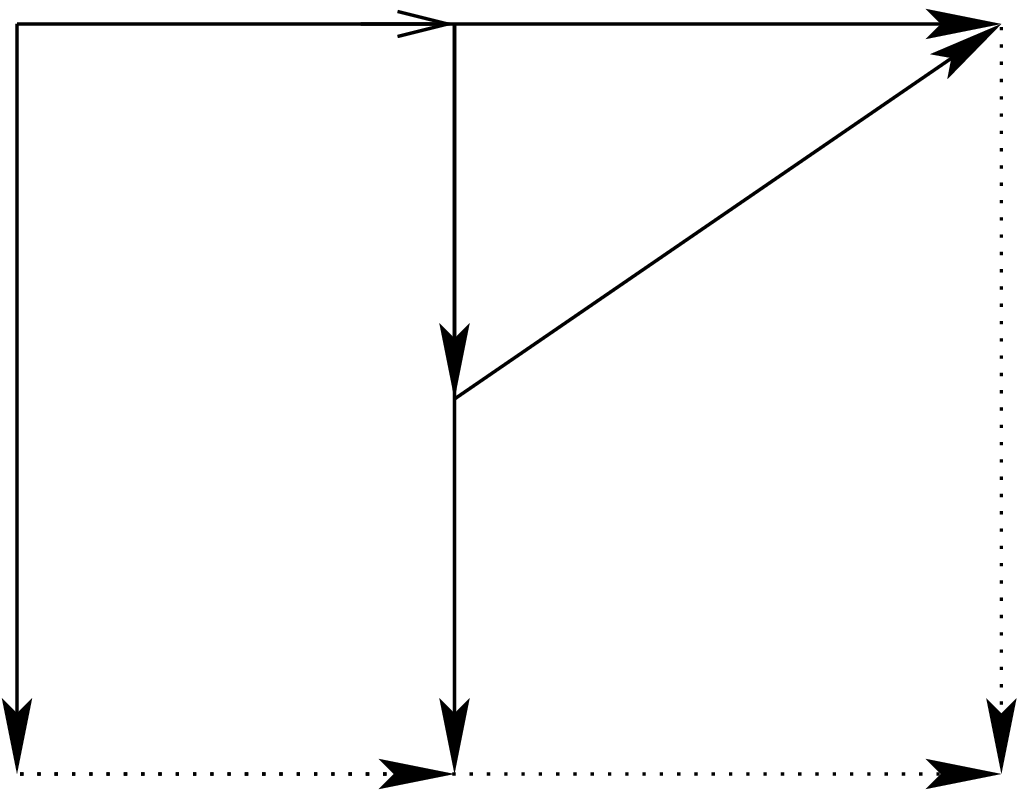}}
\put(4,6){$Z$}\put(61,6){$Z_1$}\put(126,6){$X'$} 
\put(4,106){$X$}\put(61,106){$Y_1$}\put(126,106){$Y$} 
\put(48,60){$X_1$} \put(4,60){$g$}\put(20,107){{\small$\nu_{X_0}(t_0)$}} \put(80,106){{\small $h'_0$}}\put(36,81){{\small $g'_0\!\wedge\! h'_0$}}\put(52,36){{\small$g''_0$}}\put(94,73){{\small $h_1$}}\put(128,60){{\small $k$}}
\end{picture}\caption{The decomposition of lcm's in~$\mathcal{R^{\scriptscriptstyle +}}(A^{\scriptscriptstyle +})_{A_X^{\scriptscriptstyle +}\to\cdot}$}\label{diagramme} 
\end{figure}

\begin{proof}[Proof of Proposition~\ref{rtgvfresc}] $(i)$ Let~$S$ be the atom set of~$A^{\scriptscriptstyle +}$. Let~$A_X^{\scriptscriptstyle +}$ be a parabolic submonoid of~$A^{\scriptscriptstyle +}$, and consider~$g,h$ that lie in~$\ConjplusA_{A_X^{\scriptscriptstyle +}\to\cdot}$. Since~$\ConjplusA$ is a full subcategory of~$\CCA$, it is enough to prove that the lcm and the gcd for the left-divisibility in~$A^{\scriptscriptstyle +}$ of~$g$ and~$h$ belong to~$\ConjplusA_{A_X^{\scriptscriptstyle +}\to\cdot}$. Let~$s$ lie in~$S$.\\ 
gcd's: By the second defining property of a~$\nu$-function, if~$s$ left-divides both~$g$ and~$h$ then~$\nu_X(s)$ left-divides both~$g$ and~$h$. Since~$\nu_X(s)$ lies in~$\ConjplusA$, by the noetherianity property, it follows that the gcd of~$g$ and~$h$ for the left-divisibility lies in~$\ConjplusA_{A_X^{\scriptscriptstyle +}\to\cdot}$.\\
lcm's: let~$g$ be in~$\ConjplusA_{A_X^{\scriptscriptstyle +}\to A_Z^{\scriptscriptstyle +}}$ and~$h$ be in~$\ConjplusA_{A_X^{\scriptscriptstyle +}\to A_Y^{\scriptscriptstyle +}}$. We have to prove that the lcm~$g\lor h$ of~$g$ and~$h$ belongs to~$\ConjplusA_{A_X^{\scriptscriptstyle +}\to\cdot}$. Clearly, we only need to consider the case when~$g = \nu_X(s)$ for some~$s$ in~$S$. Furthermore we can assume without restriction that~$g$ does not left-divide~$h$. Write~$g\lor h = hk$. It is enough to prove that there exist~$t$ in~$S$  and~$k_1$ in~$A^{\scriptscriptstyle +}$ such that~$ k=\nu_Y(t)k_1$. Indeed, in this case one has~$g\lor h = g\lor (h\nu_Y(t))$, and if~$k_1\neq 1$ we can write~$k_1 = \nu_{Y'}(t')k_2$; in particular,~$k = \nu_Y(t)\nu_{Y'}(t')k_2$. We can repeat the process as long as~$k_i\neq 1$. By the noetherianity property of~$A^{\scriptscriptstyle +}$ the process has to stop.  So, Consider a decomposition~$h = \nu_{X_0}(t_0)h'_0$ where~$X_0 = X$ and~$\nu_{X_0}(t_0)$ belongs to~$\ConjplusA_{A_X^{\scriptscriptstyle +}\to A_{Y_1}^{\scriptscriptstyle +}}$. Write~$g\lor \nu_{X_0}(t_0) = \nu_{X_0}(t_0)g'_0 = \nu_{X_0}(t_0)(g'_0\land h'_0)g''_0$ and~$h_0' = (g'_0\land h'_0)h_1$ (see Figure~\ref{diagramme}). Since~$A^{\scriptscriptstyle +}$ has a~$\nu$-function and~$g$ is equal to~$\nu_X(s)$, there exist~$Z_1$ and~$X_1$ such that~$g'_0$,~$g''_0$ and~$h_1$ belong to~$\ConjplusA_{A_{Y_1}^{\scriptscriptstyle +}\to A_{Z_1}^{\scriptscriptstyle +}}$,~$\ConjplusA_{A_{X_1}^{\scriptscriptstyle +}\to A_{Z_1}^{\scriptscriptstyle +}}$ and~$\ConjplusA_{A_{X_1}^{\scriptscriptstyle +}\to A_{Y}^{\scriptscriptstyle +}}$, respectively. By assumption,~$g'_0$ does not left-divide~$h'_0$. In particular~$g''_0\neq 1$, and there exists~$s_1$ in~$S$ such that~$g_1 = \nu_{X_1}(s_1)$ left-divides~$g''_0$; Furthermore,~$g''_0$ left-divides~$h_1k$.  Therefore, if~$h_1=1$ we are done. Otherwise, replacing~$g$ and~$h$ by~$g_1$ and~$h_1$, respectively, we can find~$t_1$~in~$S$ such that~$h_1 = \nu_{X_1}(t_1)h'_1$. Furthermore, we can write~$g_1\lor \nu_{X_1}(t_1) = \nu_{X_1}(t_1)g'_1 = \nu_{X_1}(t_1)(g'_1\land h'_1)g''_1$ and~$h_1' = (g'_1\land h'_1)h_2$ with~$g'_1$,~$g''_1$ and~$h_2$ in~$\ConjplusA_{A_{Y_2}^{\scriptscriptstyle +}\to A_{Z_2}^{\scriptscriptstyle +}}$,~$\ConjplusA_{A_{X_2}^{\scriptscriptstyle +}\to A_{Z_2}^{\scriptscriptstyle +}}$ and~$\ConjplusA_{A_{X_2}^{\scriptscriptstyle +}\to A_{Y}^{\scriptscriptstyle +}}$, respectively, for some subsets~$X_2$,~$Y_2$ of~$S$. Again, there exists~$s_2$ in~$S$ such that~$\nu_{X_2}(s_2)$ left-divides~$g''_1$, and~$g''_1$ left-divides~$h_2k$. We can repeat the process as long as~$h_i\neq 1$. But,~$h = \nu_{X_0}(t_0)(g'_0\land h'_0)\nu_{X_1}(t_1)(g'_1\land h'_1)\cdots \nu_{X_i}(t_i)(g'_i\land h'_i)h_{i+1}$. By the noetherianity property, there exists~$i$ such that~$h_{i+1} = 1$. In this case,~$\nu_{X_i}(s_i)$ left-divides~$k$, and we are done.\\$(ii)$ The proof is similar to the proof of~$(i)$.
\end{proof}

\begin{figure}[ht]
\begin{picture}(190,60)
\put(10,6){\includegraphics[scale = 0.6]{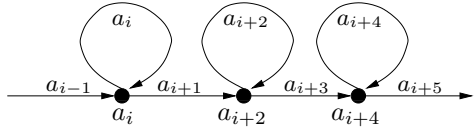}}
\put(50,0){$a_i$}\put(90,0){$a_{i+2}$}\put(133,0){$a_{i+4}$}
\put(50,36){{\small $a_i$}}\put(92,36){{\small $a_{i+2}$}}\put(135,36){{\small $a_{i+4}$}}
\put(25,11){{\small $a_{i-1}$}}\put(67,11){{\small $a_{i+1}$}}\put(115,11){{\small $a_{i+3}$}}\put(158,11){{\small $a_{i+5}$}}
\end{picture} 
\caption{The positive ribbon category of~$F_2$}
\end{figure}

From Propositions~\ref{trtrtrrt} and~\ref{rtgvfresc} we obtain Theorem~\ref{ThintroR(A)}, which applies in the case of Artin-Tits groups of spherical type. The latter groups are not the only Garside groups which have a~$\nu$-structure:
\begin{Cor} (i) Consider the notation of~Example~\ref{exef2}. The category~$\mathcal{R}^{\scriptscriptstyle +}\!(F_2^{\scriptscriptstyle +})$ is a Garside category with~$\mathcal{R}\!(F_2)$ as associated Garside groupoid.\\ 
(ii) The category~$\ConjplusA$  of the Garside group~$A$ of Example~\ref{exef23} is a Garside category with~$\ConjA$ as associated Garside groupoid.\\ 
(iii) The category~$\ConjplusA$  of the Garside group~$A$ of Example~\ref{exef24} is a Garside category with~$\ConjA$ as associated Garside groupoid. 
\end{Cor}

\begin{proof}
We only need to explicit a~$\nu$-structure in each case:\\~$(i)$ for~$X = \emptyset$ and~$i\in\mathbb{Z}$, one has~$\nu_X(a_i) = \tilde{\nu}_X(a_i) = a_i$; for~$X = \{a_j\mid j\in\mathbb{Z}\}$ and~$i\in\mathbb{Z}$, one has~$\nu_X(a_i) = \tilde{\nu}_X(a_i) = a_0a_1$; for~$X = \{a_j\}$ and~$i\in\mathbb{Z}$ with~$i\neq j$, one has~$\nu_X(a_j) = \tilde{\nu}_X(a_j) = a_{j}$,~$\nu_X(a_i) = a_{j+1}$, and~$\tilde{\nu}_X(a_i) = a_{j-1}$.\\
$(ii)$  for~$X = \emptyset$ or~$X = \{c\}$ and~$t = a,b,c$, one has~$\nu_X(t) = \tilde{\nu}_X(t) = t$; for~$X = \{a,b,c\}$ and~$t = a,b,c$, one has~$\nu_X(t) = \tilde{\nu}_X(t) = a^2c$; for~$X = \{a,b\}$, one has~$\nu_X(a) = \tilde{\nu}_X(a) = \nu_X(b) =  \tilde{\nu}_X(b) = a^2$, and~$\nu_X(c) = \tilde{\nu}_X(c) = c$.\\$(iii)$ for~$X = \emptyset$ and~$t = a,b,c$, one has~$\nu_X(t) = \tilde{\nu}_X(t) = t$; for~$X = \{a,b,c\}$ and~$t = a,b,c$, one has~$\nu_X(t) = \tilde{\nu}_X(t) = c^2$; for~$X = \{a\}$, one has~$\nu_X(a) = \tilde{\nu}_X(a) = a$,~$\nu_X(b) = \nu_X(c) = ba$ and~$\tilde{\nu}_X(b) = \tilde{\nu}_X(c) = ab$; for~$X = \{b\}$, one has~$\nu_X(b) = \tilde{\nu}-X(b) = b$,~$\nu_X(a) = \nu_X(c) = ab$ and~$\tilde{\nu}_X(a) = \tilde{\nu}_X(c) = ba$. 
\end{proof}
\section{Presentation of groupoids}
\label{sectionpresgrp}
Given a group, it is a classical question to obtain a presentation of this group. Here, we consider the similar problem in the context of groupoids. For all this section we fix a monoidal Garside structure~($A^{\scriptscriptstyle +},\Delta)$, such that the Garside monoid~$A^{\scriptscriptstyle +}$ has a~$\nu$-structure. We denote by~$S$ the atom set of~$A^{\scriptscriptstyle +}$ and by~$A$ its group of fractions. In the previous section, we have seen that the ribbon groupoid~$\ConjA$ is a Garside groupoid. Let us first precise the notion of a groupoid presentation. We recall that the free category~$\mathcal{C}(\Gamma)$ associated with a graph~$\Gamma$ has been defined in Section~\ref{definsmcat}.
\begin{Def}[Generating quiver]
Let~$\mathcal{C}$ be a small category. Let~$\Gamma$ be a quiver whose vertex set is the object set of~$\mathcal{C}$ and whose edges are morphisms of~$\mathcal{C}$.\\
(i)  We say that~$\Gamma$ \emph{generates~$\mathcal{C}$ as a category} if the canonical map from~$\Gamma$ to~$\mathcal{C}$ extends to a surjective functor from the free category~$\mathcal{C}(\Gamma)$ to the category~$\mathcal{C}$.\\
(ii) If~$\mathcal{C}$ is a groupoid, we say that~$\Gamma$ \emph{generates~$\mathcal{C}$ as a groupoid} if the canonical map from~$\Gamma$ to~$\mathcal{C}$ extends to a surjective functor from the free groupoid~$\mathcal{G}(\mathcal{C}(\Gamma))$ to the groupoid~$\mathcal{C}$.
\end{Def}
\begin{Def}[Groupoid presentation] Let~$\mathcal{C}$ be a small category. Consider a quiver~$\Gamma$ and a set~$R = \{(g_i,h_i), i\in I\}$ of morphisms of~$\mathcal{C}(\Gamma)$ with the same source and the same target. Denote by~$\equiv_R$ is the congruence on~$\mathcal{M}(\Gamma)$ generated by~$R$.\\ 
(i) We say that~$\langle\Gamma,R\rangle$ is a \emph{presentation of~$\mathcal{C}$ as a category} if~$\Gamma$ generates~$\mathcal{C}$ as a category and the surjective functor from~$\mathcal{C}(\Gamma)$ to~$\mathcal{C}$ induces an isomorphism of categories from~$\mathcal{C}(\Gamma,\equiv_R)$ to~$\mathcal{C}$.\\
(ii) When~$\mathcal{C}$ is a groupoid, we say that~$\langle\Gamma,R\rangle$ is a \emph{presentation of~$\mathcal{C}$ as a groupoid} if~$\Gamma$ generates~$\mathcal{C}$ as a groupoid and the surjective functor from~$\mathcal{G}(\mathcal{C}(\Gamma))$ to~$\mathcal{C}$ induces an isomorphism of groupoid from~$\mathcal{G}(\mathcal{C}(\Gamma,\equiv_R))$ to~$\mathcal{C}$.\end{Def}

 In the sequel, we write~$g_i = h_i$ for~$(g_i,h_i)$ in the category/groupoid presentations. Our main objective in this section is to obtain a groupoid presentation of the groupoid of~$\ConjA$. 

\begin{Lem} \label{generquiverA}Let~$\mathcal{A}$ be the graph whose vertices are the standard parabolic subgroups of~$A$, and whose edge set is the image of the~$\nu$-functions: to each element~$\nu_X(s)$ that lies in~$\ConjplusA_{A_X\to A_Y}$ corresponds an edge from~$A_X$ to~$A_Y$ in~$\mathcal{A}$. The quiver~$\mathcal{A}$ generates~$\ConjplusA$ as a category and~$\ConjA$ as a groupoid.
 \end{Lem}
The reader should remember that for~$s\neq t$ we can have~$\nu_X(s) = \nu_X(t)$; in that case, the pairs~$(X,s)$ and~$(X,t)$ correspond to the same edge in the quiver~$\mathcal{A}$. 
\begin{proof} The atom graph of~$\ConjplusA$ is a generating quiver and the image of the~$\nu$-functions is precisely the atom set of the category~$\ConjplusA$ (see Section~\ref{sectiongardgrd}). Furthermore,~$\ConjA$ is the group of fractions of the Garside category~$\ConjplusA$.  \end{proof}

In the case of an Artin-Tits group of spherical type, we have seen in Example~\ref{exemplnufunctspherATGP} that there is two types of atoms in~$\ConjA$, namely the~$\Delta_{X_i}$ and the~$\Delta_X^{-1}\Delta_{X\cup\{s\}}$. We are going to see that this is a general fact; we need first to study the quasi-centralizer of a Garside group.
\subsection{The quasi-centralizer} 
\begin{Def}[Quasi-centralizer]
Let~$G^{\scriptscriptstyle +}$ be a Garside monoid with~$G$ as Garside group. The \emph{quasi-centralizer} of the monoid~$G^{\scriptscriptstyle +}$ is the submonoid~$\QZ(G^{\scriptscriptstyle +})$ defined by $$\QZ(G^{\scriptscriptstyle +}) = \{g\in G^{\scriptscriptstyle +}\mid gG^{\scriptscriptstyle +} = G^{\scriptscriptstyle +}g\};$$ The \emph{quasi-centralizer} of the group~$G$ is the subgroup~$\QZ(G)$ defined by $$\QZ(G) = \{g\in G\mid gG^{\scriptscriptstyle +} = G^{\scriptscriptstyle +}g\}.$$ In other words,~$\QZ(G^{\scriptscriptstyle +})$ and~$QZ(G)$ are equal to~$\mathcal{R}^{\scriptscriptstyle +}_{G^{\scriptscriptstyle +}\to G^{\scriptscriptstyle +}}$ and~$\mathcal{R}_{G\to G}$, respectively.
\end{Def}

By definition, the Garside element~$\Delta$ belongs to~$\QZ(G^{\scriptscriptstyle +})$. In this section, we extend the main result of~\cite{Pic2} into every Garside group. The proof is in the spirit of the latter reference. However, some definitions and technical arguments of the proof of~\cite{Pic2} cannot be simply extended to every Garside group. This is because in general the set of divisors of~$\Delta$  is not finite. In our proof, this argument is replaced by the noetherianity property of~$G^{\scriptscriptstyle +}$. We need some preliminary results.
\begin{Lem}\label{defindesDeltag} Let~$(G^{\scriptscriptstyle +},\Delta)$ be a monoidal Garside structure.\\
(i) For every~$g$ in~$G^{\scriptscriptstyle +}$ the set~$\{h^{-1}(h\lor g)\mid h\in G^{\scriptscriptstyle +}\}$  has a lcm~$\Delta_g$ for the left-divisibility. Furthermore,~$g$ left-divides~$\Delta_g$  and~$G^{\scriptscriptstyle +}\Delta_g\subseteq gG^{\scriptscriptstyle +}$.\\
(ii) For every~$g$ in~$G^{\scriptscriptstyle +}$ the set~$\{(h\tilde{\lor} g)h^{-1}\mid h\in G^{\scriptscriptstyle +}\}$ has a~$\tilde{\Delta}_g$ for the right-divisibility. Furthermore,~$g$ right-divides~$\tilde{\Delta}_g$ and~$\tilde{\Delta}_gG^{\scriptscriptstyle +}\subseteq G^{\scriptscriptstyle +}g$.
\end{Lem}
\begin{proof} By symmetry, it is enough to prove~$(i)$. From the noetherianity property of~$G^{\scriptscriptstyle +}$ it follows that every none-empty subset of~$G^{\scriptscriptstyle +}$ that has a common multiple for the left-divisibility has a lcm for the left-divisibility. But~$\Delta$ is quasi-central and there exists a positive integer~$n$ such that~$g$ left-divides~$\Delta^n$. Therefore,~$\Delta^n$ is a common multiple of~$\{h^{-1}(h\lor g)\mid h\in G^{\scriptscriptstyle +}\}$, and therefore~$\Delta_g$ exists. By definition we have~$G^{\scriptscriptstyle +}\Delta_g\subseteq gG^{\scriptscriptstyle +}$. Finally, considering~$h = 1$, we obtain that~$g$ left-divides~$\Delta_g$. 
 \end{proof}
\begin{Lem}\label{lemsuitetaun}
 Let~$(G^{\scriptscriptstyle +},\Delta)$ be a monoidal Garside structure and~$g$ lie in~$G^{\scriptscriptstyle +}$.\\ Define~$(\tau(n))_{n\in\mathbb{N}}$ by~$\tau(0) = g$, ~$\tau(2n+1) = \Delta_{\tau(2n)}$ and~$\tau(2n+2) = \tilde{\Delta}_{\tau(2n+1)}$. The sequence~$(\tau(n))_{n\in\mathbb{N}}$ stabilizes to an element~$\tau_g$ which belongs to~$QZ(G^{\scriptscriptstyle +})$. Furthermore, for every element~$h$ in~$QZ(G^{\scriptscriptstyle +})$,~$g$ left-divides~$h$ if and only if~$\tau_g$ left-divides~$h$.\label{defindestaug} In particular, for every quasi-central element~$h$, we have~$\tau_h = h$. 
\end{Lem}
When the set of divisors of~$\Delta$ is finite, then one has~$\tau_g =\tau(1) = \Delta_g$~\cite{Pic2}.
\begin{proof} By the previous lemma,~$\tau(2n+1)$ left-divides~$\tau(2n)$, and~$\tau(2n+2)$ right-divides~$\tau(2n+1)$. Furthermore, with the notation of Lemma~\ref{defindesDeltag}'s proof, all the~$\tau_n$ divides~$\Delta^n$ (for both the left-divisibility and the right-divisibility since~$\Delta^n$ is a quasi-central element). By the noetherianity property, there exist~$\Delta_g$ in~$G^{\scriptscriptstyle +}$ and~$N\in\mathbb{N}$ such that~$\tau(n)= \tau_g$ for every~$n\geq N$. By Lemma~\ref{defindesDeltag}, we get the sequence of inclusions~$\tau(2n+2)G^{\scriptscriptstyle +}\subseteq G^{\scriptscriptstyle +}\tau(2n+1)\subseteq \tau(2n)G^{\scriptscriptstyle +}$. Thus,~$\tau_g$ is a quasi-central element. Assume~$h$ lies in~$QZ(G^{\scriptscriptstyle +})$ such that~$g$ left-divides~$h$, and write~$h = gg'$. By assumption, we have~$G^{\scriptscriptstyle +}h = hG^{\scriptscriptstyle +} = gg'G^{\scriptscriptstyle +}$. Therefore,~$h$ is a common multiple of the set~$\{k^{-1}(k\lor g)\mid k\in G^{\scriptscriptstyle +}\}$ and~$\Delta_g$ left-divides~$h$. But~$h$ is quasi-central, then~$\Delta_g$, that is~$\tau(1)$, also right-divides~$h$. Inductively we get that~$\tau(n)$ left-divides and right-divides~$h$ for every~$n$. Finally, by Lemma~\ref{defindesDeltag},~$\tau(2n)$ left-divides~$\tau(2n+1)$ and~$\tau(2n+1)$ right-divides~$\tau(2n+2)$ for every~$n$; then we can write~$\tau_g = g_1gg_2$ for some~$g_1,g_2$ in~$G^{\scriptscriptstyle +}$. Since~$\tau_g$ is a quasi-central, it follows that~$g$ left-divides~$\tau_g$. 
\end{proof}
\begin{Rem} \label{remlemsuitetaun}With the notation of Lemma~\ref{lemsuitetaun} and a similar proof of this lemma, we can obtain an element~$\tilde{\tau}_g$ by considering the sequence~$(\tilde{\tau}(n))_{n\in\mathbb{N}}$ defined by~$\tilde{\tau}(0) = g$, ~$\tilde{\tau}(2n+1) = \tilde{\Delta}_{\tilde{\tau}(2n)}$ and~$\tilde{\tau}(2n+2) = \Delta_{\tau(2n+1)}$. Using that an element of~$\QZ(G^{\scriptscriptstyle +})$ has the same set of left-divisors and of right-divisors, we deduce that~$\tau_g = \tilde{\tau}_g$ for every~$g$ in~$G^{\scriptscriptstyle +}$. 
\end{Rem}

\begin{Prop} \label{centrcommprop}Let~$(G^{\scriptscriptstyle +},\Delta)$ be a monoidal Garside structure, and denote by~$S$ the atom set of~$G^{\scriptscriptstyle +}$. For every~$s,t$ in~$S$, either~$\tau_s = \tau_t$ or~$\tau_s\!\wedge\! \tau_t = \tau_s\tilde{\wedge} \tau_t  =  1$ in~$G^{\scriptscriptstyle +}$. Furthermore, in any case,~$\tau_s\tau_t = \tau_t\tau_s$.
\end{Prop}
\begin{proof}
The main argument of the first part's proof is as in~\cite{Pic2}. Since~$\tau_t$ is a quasi-central element, it has the same left-divisors and right-divisors. Assume that~$s$ divides~$\tau_t$ and let us prove that~$\tau_s = \tau_t$. By Lemma~\ref{defindestaug}, we can write~$\tau_t = h\tau_s = gt$. Still by Lemma~\ref{defindestaug}, if~$t$ divides~$\tau_s$, then~$\tau_t$ divides~$\tau_s$ and we are done:~$\tau_s = \tau_t$. Assume this is not the case. Since~$t$ is an atom of~$G^{\scriptscriptstyle +}$, the elements~$g$ and~$\tau_s$ are prime to each other for the right-divisibility in~$G^{\scriptscriptstyle +}$. This implies that~$\tau_t$ is the lcm of~$g$ and~$h$ for the left-divisibility. In particular,~$t = g^{-1}(h\lor g)$. Then, by definition,~$t$ left-divides~$\Delta_h$ and, therefore, left-divides~$\tau_h$. In particular, we can write~$\tau_h = \tau_tk$ for some~$k$ in~$G^+$. But~$\tau_s$ and~$\tau_t$ are quasi-central elements, then~$h$ is also a quasi-central element and~$h = \tau_h$. A contradiction: we get~$\tau_t = h\tau_s = \tau_h\tau_s = \tau_tk\tau_s$ with~$\tau_s\neq 1$. Now, if~$\tau_s\!\wedge\! \tau_t \neq 1$ then there exists~$u$ in $S$ that divides both~$\tau_s$ and~$\tau_t$. By the above argument, we have~$\tau_s = \tau_u = \tau_t$. 

Assume finally that~$\tau_s\neq\tau_t$. We can write~$\tau_s\tau_t = \tau_tg_1 = g_2\tau_s$ and~$\tau_t\tau_s = \tau_sh_1 = h_2\tau_t$ where~$g_1,g_2,h_1,h_2$ lie in~$G^{\scriptscriptstyle +}$. Clearly~$g_1,g_2,h_1,h_2$ have to be quasi-central elements. As~$\tau_t\neq\tau_s$, by the first part of the proof, the decomposition~$\tau_t^{-1}\tau_s$  is a left greedy normal form ({\it cf.} Prop.~\ref{uniquedecomposprop}$(ii)$). But we have~$\tau_t^{-1}\tau_s = g_1\tau_t^{-1} = \tau_sh_1^{-1}$. Since~$\tau_t^{-1}\tau_s$ is not in~$G^{\scriptscriptstyle +}$,~$\tau_t$ and~$\tau_s$ do not right-divide~$g_1$ and~$h_1$, respectively. This implies that both~$g_1\tau_t^{-1}$  and~$\tau_sh_1^{-1}$ are right greedy normal forms. By unicity of this normal form, we get~$g_1 = \tau_s$ and~$h_1 = \tau_t$. Thus,~$\tau_s\tau_t = \tau_t\tau_s$. \end{proof}

\begin{The} \label{theoremcentralisateurgeneralisationpicantin}Let~$(G^{\scriptscriptstyle +},\Delta)$ be a monoidal Garside structure, and denote by~$S$ the atom set of~$G^{\scriptscriptstyle +}$. \\(i) The monoid~$QZ(G^{\scriptscriptstyle +})$ is a free commutative monoid. Furthermore there exists a projection~$\tau: G^{\scriptscriptstyle +}\to \QZ(G^{\scriptscriptstyle +}),\ g\mapsto \tau_g$ such that
\begin{enumerate}
\item the set~$\{\tau_s\mid s\in S\}$ is a free base of~$\QZ(G^{\scriptscriptstyle +})$ (where some~$\tau_s$ are possibly equal);
\item the map~$\tau$ respects the left-divisibility and the right-divisibility, in particular if~$s\in S$ left/right-divides~$g\in\QZ(G^{\scriptscriptstyle +})$ then~$\tau_s$ left/right-divides~$g$; 
\item the map~$\tau$ is a semilattice homomorphism for~$\lor$ and for~$\tilde{\lor}$; in particular,~$\tau_{g\lor h} = \tau_g\lor \tau_h$ for every~$g,h$ in~$G^{\scriptscriptstyle +}$. \end{enumerate}
(ii) The group~$QZ(G)$ is a free commutative group with~$\{\tau_s\mid s\in S\}$  as a free base (where some~$\tau_s$ are possibly equal).
\end{The}
This theorem should be compared with the definition of~$\nu$-functions ({\it cf.} Definition~\ref{defnustructure}). Note that the map~$\tau$ is not a semilattice homomorphism for~$\land$ (nor~$\tilde{\land}$ ) in general~\cite{Pic2}. 
\begin{proof}
$(i)$ We already know by Lemma~\ref{defindestaug} that~$\tau(\QZ(G^{\scriptscriptstyle +})) =\QZ(G^{\scriptscriptstyle +})$, that~$\tau^2 = \tau$ and that~$\tau$ respects the left-divisibility and the right-divisibility.  Furthermore, using the noetherianity property and the same lemma, we get that the set~$\{\tau_s\mid s\in S\}$ is a generating set for~$\QZ(G^{\scriptscriptstyle +})$. By Proposition~\ref{centrcommprop}, it follows that the monoid~$\QZ(G^{\scriptscriptstyle +})$ is commutative, and that for distinct elements~$\tau_s$ and~$\tau_t$, the element~$\tau_s\tau_t$ is a common multiple of~$\tau_s$ and~$\tau_t$ for both the left-divisibility and the right-divisibility. Using Lemma~\ref{defindestaug}, we get that~$\tau_s\tau_t$ is the lcm of~$\tau_s$ and~$\tau_t$ (for both the left-divisibility and the right-divisibility): otherwise, there would be some~$u$ in~$S$ that divides both~$\tau_s$ and~$\tau_t$ and then we should have~$\tau_s = \tau_u = \tau_t$.  Therefore,~$QZ(G^{\scriptscriptstyle +})$ is a free commutative monoid with~$\{\tau_s\mid s\in S\}$ as a free base. Now, let~$g$ and~$h$ belong to~$G^{\scriptscriptstyle +}$. They left-divides~$g\lor h$. Therefore~$\tau_g$,~$\tau_h$ and~$\tau_g\lor \tau_h$ left-divides~$\tau_{g\lor h}$. But~$\tau_g$ and~$\tau_h$ have a common multiple in~$\QZ(G^{\scriptscriptstyle +})$ for the left-divisibility (that is some power of~$\Delta$). By the noetherianity property, there exists a minimal common multiple~$\tau_gg' = \tau_hh'$ of~$\tau_s$ and~$\tau_t$ for the left-divisibility in~$QZ(G^{\scriptscriptstyle +})$. Assume there exists~$u$ in~$S$ that right-divides both~$g'$ and~$h'$; then by Remark~\ref{remlemsuitetaun},~$\tau_u$ also right-divides~$g'$ and~$h'$, which is impossible by minimality. Therefore~$\tau_g\lor \tau_h$ belongs to~$\QZ(G^{\scriptscriptstyle +})$. Since~$g\lor h$ left-divides~$\tau_g\lor \tau_h$, the element~$\tau_{g\lor h}$ left-divides~$\tau_g\lor \tau_h$. As a conclusion,~$\tau_{g\lor h}$ is equal to~$\tau_g\lor \tau_h$. \\ 
$(ii)$ Let~$g$ lie in~$\QZ(G)$ and~$a^{-1}b$ be its left greedy normal form ({\it cf.} Proposition~\ref{uniquedecomposprop}). We claim that~$a$ and~$b$ belong to~$\QZ(G^{\scriptscriptstyle +})$. This proves that~$\QZ(G)$ is (isomorphic to) the group of fractions of~$\QZ(G^{\scriptscriptstyle +})$ and, therefore, the second point of the proposition. Actually, every element~$g$ of~$G$ can be written  as~$\Delta^{-n}g^{\scriptscriptstyle +}$ with~$n\in \mathbb{N}$ and~$g^{\scriptscriptstyle +}$ in~$G^{\scriptscriptstyle +}$. But~$\Delta$ is a quasi-central element. Therefore, when~$g$ lies in~$\QZ(G)$, then~$g^{\scriptscriptstyle +}$ and~$\Delta^{n}$ belong to~$\QZ(G^{\scriptscriptstyle +})$. Now assume that~$g = c^{-1}d$ where both~$c$ and~$d$ lie in~$\QZ(G^{\scriptscriptstyle +})$. If some~$u$ in~$S$ left-divides both~$c$ and~$d$, then~$\tau_u$ also left-divides~$c$ and~$d$. We conclude as in the proof of~$(i)$.     
\end{proof}

In the next section, we are going to apply the above results to parabolic submonoids. We recall that we have fixed a Garside monoid~$A^{\scriptscriptstyle +}$ with a~$\nu$-structure in the introduction of Section~\ref{sectionpresgrp}. 

\begin{Cor}\label{cornuspeciaux} Let~$G^{\scriptscriptstyle +} = A^{\scriptscriptstyle +}_X$ be a parabolic submonoid of~$A^{\scriptscriptstyle +}$. Let~$s$ belong to~$S$.\\(i) If~$s$ lies in~$X$, then~$\tau_s = \nu_X(s)$;\\ (ii) If~$\nu_X(s)$ belongs to~$A_X^{\scriptscriptstyle +}$, then there exists~$t$ in~$X$ such that~$\nu_X(s) = \tau_t$.  
\end{Cor}
\begin{proof} $(i)$ Assume~$s$ lies in~$X$. Since~$s$ left-divides~$\tau_s$, by the defining properties of a~$\nu$-structure,~$\nu_X(s)$ left-divides~$\tau_s$. But~$\tau_s$ belongs to~$A_X^{\scriptscriptstyle +}$ which is a parabolic submonoid. Therefore,~$\nu_X(s)$ belongs to~$A^{\scriptscriptstyle +}_X$. Consequently,~$\nu_X(s)$ lies in~$\QZ(A_X^{\scriptscriptstyle +})$, and has to be equal to~$\tau_s$.\\
$(ii)$  Assume~$\nu_X(s)$ belongs to~$A_X^{\scriptscriptstyle +}$. This implies that~$\nu_X(s)$ lies in~$\QZ(A_X^{\scriptscriptstyle +})$, and therefore is equal to a product~$\tau_{t_1}\cdots\tau_{t_k}$ with~$t_1,\ldots,t_k$ in~$S$. But~$\nu_X(s)$ is an atom of~$\ConjplusA$, hence~$k = 1$. \end{proof}
\subsection{Defining relations}
We still consider the general notations of Section~\ref{sectionpresgrp}. In Lemma~\ref{generquiverA}, we have seen that the quiver~$\mathcal{A}$ whose edges are the~$\nu_X(s)$ is a generating quiver for the groupoid~$\ConjA$. In Corollary~\ref{cornuspeciaux}, we have seen that among the~$\nu_X(s)$ there is two kinds of atoms of~$\ConjplusA$, like in the case of Artin-Tits groups of spherical type ({\it cf.} Example~\ref{exemplnufunctspherATGP}). In the sequel, for every parabolic submonoid~$G^{\scriptscriptstyle +} = A_X^{\scriptscriptstyle +}$ of~$A^{\scriptscriptstyle +}$, and every element~$s$ that belongs to~$S$,  we denote by~$\tau(X,s)$ the element~$\nu_X(s)$ when it belongs to~$A_X^{\scriptscriptstyle +}$, and by~$\nu(X,s)$ otherwise. Similarly, we write~$\tilde{\tau}(X,s)$ or~$\tilde{\nu}(X,s)$ for~$\tilde{\nu}_X(s)$, depending whether or not this atom lies in~$A_X^{\scriptscriptstyle +}$. For instance, for every $s$ in $X$, one has $\tau_s = \tau(X,s) = \nu_X(s)$.
\begin{Lem}\label{lemfghbvcasc} Let~$A_X^{\scriptscriptstyle +}$ and~$A_Y^{\scriptscriptstyle +}$ be parabolic submonoids of~$A^{\scriptscriptstyle +}$. Assume~$\tau(X,s)$ lies in~$\QZ(A_X^{\scriptscriptstyle +})$ and~$\nu(X,t)$ belongs to~$\ConjplusA_{A_X^{\scriptscriptstyle +}\to A_Y^{\scriptscriptstyle +}}$. Then, there exists~$\tau(Y,s')$ in~$\QZ(A_Y^{\scriptscriptstyle +})$ such that \begin{equation} \tau(X,s)\nu(X,t) = \nu(X,t)\tau(Y,s'). \end{equation}
\end{Lem}
\begin{proof}
Since~$\tau(X,s)$ lies in~$\QZ(A_X^{\scriptscriptstyle +})$ and~$\nu(X,t)$ belongs to~$\ConjplusA_{A_X^{\scriptscriptstyle +}\to A_Y^{\scriptscriptstyle +}}$, the element~$\nu(X,t)^{-1}\tau(X,s)\nu(X,t)$ belongs to~$\QZ(A_Y^{\scriptscriptstyle +})$. But~$\tau(X,s)$ is an atom of the category~$\ConjplusA$. Therefore~$\nu(X,t)^{-1}\tau(X,s)\nu(X,t)$ is also an atom of~$\ConjplusA$. By Corollary~\ref{cornuspeciaux}, we are done. 
\end{proof}
When proving Lemma~\ref{lemsecdecpodeslcmdesnu} and Theorem~\ref{theorpreseconjplusA} we will use the following result:
\begin{Lem} \label{mpolkiujyh}Let~$A_X^{\scriptscriptstyle +}$ and~$A_Y^{\scriptscriptstyle +}$ be  parabolic submonoids of~$A^{\scriptscriptstyle +}$. Assume  that~$g$ is equal to~$\nu(X_1,s_1)\cdots \nu(X_k,s_k)$ and lies in~$\ConjplusA_{A_X^{\scriptscriptstyle +}\to A_Y^{\scriptscriptstyle +}}$, where~$X_1 = X$; assume furthermore that~$h_1$ belongs to~$A_X^{\scriptscriptstyle +}$ and~$h_2$ belongs to~$A^{\scriptscriptstyle +}_Y$. Then,\\
(i) the lcm and the gcd in~$A^{\scriptscriptstyle +}$ of~$g$ and~$h_1$ for the left-divisibility is~$h_1g$ and~$1$, respectively.\\
(ii) the lcm and the gcd in~$A^{\scriptscriptstyle +}$ of~$g$ and~$h_2$ for the right-divisibility is~$gh_2$ and~$1$, respectively.
\end{Lem}
\begin{proof} Assume~$k = 1$. If~$u$ in~$S$ left-divides both~$g$ and~$h_1$, then~$u$ lies in~$X$ and~$\tau(X,u)$ left-divides~$\nu(X,s)$. This is impossible because in this case we would have~$\nu(X,s) = \tau(X,u)$. Therefore, the gcd in~$A^{\scriptscriptstyle +}$of~$g$ and~$h_1$ for the left-divisibility is~$1$. Similarly, the gcd of~$g$ in~$A^{\scriptscriptstyle +}$ and~$h_2$ for the right-divisibility is~$1$ (We recall that~$\nu(X,t)$ is equal to some~$\tilde{\nu}(Y,s)$). Let~$h'_2$ in~$A^{\scriptscriptstyle +}_X$ be such that~$h'_2g= gh_2$. The gcd of~$h'_2$ and~$g$ for the left-divisibility is~$1$ by the first part of the proof. Thus,~$gh_2$ is the lcm of~$g$ and~$h_2$ for the right-divisibility. Similarly,~$h_1g$ is the lcm in~$A^{\scriptscriptstyle +}$ of~$g$ and~$h_1$ for the left-divisibility.

Assume~$k\geq 2$. By an easy induction on~$k$ we deduce from the case~$k = 1$ that the element~$h_1g$ is the lcm of~$g$ and~$h_1$ for the left-divisibility. Similarly~$gh_2$ is the lcm of~$g$ and~$h_2$ for the right-divisibility. Finally, we deduce from the latter result the last part of the lemma: let~$h'_1$ be the gcd of~$g$ and~$h_1$ for the left-divisibility. Since~$A_X^{\scriptscriptstyle +}$ is a parabolic submonoid,~$h'_1$ belongs to~$A_X^{\scriptscriptstyle +}$. In particular the lcm of~$h'_1$ and~$g$ for the left-divisibility, which is~$g$, has to be equal to~$h'_1g$. Then~$h'_1 = 1$. Similarly, the gcd of~$g$ and~$h_2$ for the right-divisibility is~$1$.
\end{proof}

We recall that the quiver~$\mathcal{A}$ has been defined in Lemma~\ref{generquiverA}, and that we can have~$\nu_X(s) = \nu_X(t)$ in the quiver~$\mathcal{A}$, even though~$s\neq t$. 

\begin{Lem}\label{lemsecdecpodeslcmdesnu} Let~$A_X^{\scriptscriptstyle +}$ be a parabolic submonoid of~$A^{\scriptscriptstyle +}$ and consider~$A_Y^{\scriptscriptstyle +}$,~$A_Z^{\scriptscriptstyle +}$ two parabolic submonoids of~$A^{\scriptscriptstyle +}$ which are distinct from~$A^{\scriptscriptstyle +}_X$. Assume~$\nu(X,s)$ and~$\nu(X,t)$ belong to~$\ConjplusA_{A_X^{\scriptscriptstyle +}\to A_Y^{\scriptscriptstyle +}}$ and~$\ConjplusA_{A_X^{\scriptscriptstyle +}\to A_Z^{\scriptscriptstyle +}}$, respectively. Then, there exist a parabolic submonoid~$A^{\scriptscriptstyle +}_{X'}$, a path~$\nu(Y_1,s_1)\cdots \nu(Y_k,s_k)$ in~$\mathcal{C}(\mathcal{A})_{A^{\scriptscriptstyle +}_Y\to A^{\scriptscriptstyle +}_{X'}}$ and a path $\nu(Z_1,t_1)\cdots \nu(Z_\ell,t_\ell)$ in~$\mathcal{C}(\mathcal{A})_{A^{\scriptscriptstyle +}_Z\to A^{\scriptscriptstyle +}_{X'}}$ such that \begin{equation}\label{lemsecdecpodeslcmdesnu2} \nu(X,s)\lor \nu(X,t) =  \nu(X,s)\nu(Y_1,s_1)\cdots \nu(Y_k,s_k) = \nu(X,t)\nu(Z_1,t_1)\cdots \nu(Z_\ell,t_\ell). \end{equation}
\end{Lem}
\begin{proof} We can assume~$\nu(X,s)\neq\nu(X,t)$. The category~$\ConjplusA$ is a Garside category with~$\mathcal{A}$ for atom graph. Then we can decompose~$\nu(X,s)\lor \nu(X,t)$ in~ the category~$\ConjplusA$: $$\nu(X,s)\lor \nu(X,t) = \nu(X,s)\nu_{Y_1}(s_1)\cdots \nu_{Y_k}(s_k) = \nu(X,t)\nu_{Z_1}(t_1)\cdots \nu_{Z_\ell}(t_\ell).$$ By Lemma~\ref{lemfghbvcasc}, we can assume without restriction that there exists~$k_1$ and~$\ell_1$  such that~$\nu_{Y_i}(s_i)$ is equal to~$\nu(Y_i,s_i)$ if and only if~$i\leq k_1$, and that~$\nu_{Z_i}(t_i)$ is equal to~$\nu(Z_i,t_i)$ if and only if~$i\leq \ell_1$. Let us set~$g = \nu_{Y_{k_1+1}}\!(s_{k_1+1})\cdots \nu_{Y_k}(s_k)$ and~$h = \nu_{Z_{\ell_1+1}}(t_{\ell_1+1})\cdots \nu_{Z_\ell}(t_\ell)$, and consider~$g_1,h_1$ such that~$g\tilde{\lor} h = g_1g = h_1h$. Since~$g$ and~$h$ belong to~$A_Y^{\scriptscriptstyle +}$, the elements~$g_1$ and~$h_1$ lie in~$A_Y^{\scriptscriptstyle +}$. But~$g_1$ has to right-divide~$\nu(X,s)\nu_{Y_1}(s_1)\cdots \nu_{Y_{k_1}}(s_{k_1})$. Then, by Lemma~\ref{mpolkiujyh}, we have~$g_1 = 1$. Similarly,~$h_1 = 1$ that is~$g = h$. Finally we get~$g = h = 1$ since the lcm of $g$ and $h$ for the right-divisibility has to be~$1$.
 \end{proof}
 
\begin{The}\label{theorpreseconjplusA}  Let~$(A^{\scriptscriptstyle +},\Delta)$ be a monoidal Garside structure with a $\nu$-structure. Denote by~$S$ the atom set of~$(A^{\scriptscriptstyle +}$, and by~$A$ its group of fractions. Consider the  set~$R$ of equalities of paths in~$\mathcal{C}(\mathcal{A})$:
\begin{enumerate}
\item~$\tau(X,s)\tau(X,t) = \tau(X,t)\tau(X,s)$;\label{relundelaprese} 
\item~$\tau(X,s)\nu(X,t) = \nu(X,t)\tau(Y,s')$; \label{reldeuxdelaprese}
\item~$\nu(X,s)\nu(Y_1,s_1)\cdots \nu(Y_k,s_k) = \nu(X,t)\nu(Z_1,t_1)\cdots \nu(Z_\ell,t_\ell)$\label{deregalite}
\end{enumerate}
that hold in~$\ConjplusA$ and such that in Equalities of type~{\rm(\ref{deregalite})} the two paths correspond to two decompositions of~$\nu(X,s)\lor \nu(X,t)$ as in Equation~{\rm (\ref{lemsecdecpodeslcmdesnu2})} of Lemma~\ref{lemsecdecpodeslcmdesnu}.\\ 
Then,~$\langle \mathcal{A},R\rangle$ is a presentation of~$\ConjA$ as a groupoid.
\end{The}
\begin{proof} The Garside group~$A$ has a~$\nu$-function, then~$\ConjA$ is a Garside groupoid, with~$\ConjplusA$ as Garside category. Therefore, it is enough to prove that~$\langle \mathcal{A},R\rangle$ is a presentation of~$\ConjplusA$ as a category. We already know that~$\mathcal{A}$ is a generating quiver of~$\ConjplusA$. Now, let~$\equiv$ be the congruence on~$\mathcal{M}(\mathcal{A})$  defined by the relations considered in the theorem. Let~$\nu_{Y_1}(s_1)\cdots \nu_{Y_k}(s_k)$ and~$\nu_{Z_1}(t_1)\cdots \nu_{Z_\ell}(t_\ell)$ be two paths in~$\mathcal{C}(\mathcal{A})_{A_X^{\scriptscriptstyle +}\to A^{\scriptscriptstyle +}_{X'}}$. Clearly, if the two paths are~$\equiv$-equivalent, they represent the same element in~$\ConjplusA_{A_X^{\scriptscriptstyle +}\to A^{\scriptscriptstyle +}_{X'}}$. Conversely, assume they represent the same element in~$\ConjplusA_{A_X^{\scriptscriptstyle +}\to A^{\scriptscriptstyle +}_{X'}}$ and let us prove that they are~$\equiv$-equivalent. Using defining relations~(\ref{reldeuxdelaprese}) of~$\equiv$, we obtain a path $\nu(Y'_1,s'_1)\cdots \nu(Y'_{k_1},s'_{k_1})\tau(X',s'_{k_1+1})\cdots \tau(X',s'_{k})$ which is ~$\equiv$-equivalent to the path~$\nu_{Y_1}(s_1)\cdots \nu_{Y_k}(s_k)$. Similarly, we also obtain a path~$\nu(Z'_1,t'_1)\cdots \nu(Z'_{\ell_1},t'_{\ell_1})\tau(X',t'_{\ell_1+1})\cdots \tau(X',t'_{\ell})$ which is ~$\equiv$-equivalent to the path~$\nu_{Z_1}(t_1)\cdots \nu_{Z_\ell}(t_\ell)$. By Lemma~\ref{mpolkiujyh} the equalities~$\nu(Y'_1,s'_1)\cdots \nu(Y'_{k_1},s'_{k_1}) = \nu(Z'_1,t'_1)\cdots \nu(Z'_{\ell_1},t'_{\ell_1})$ and $\tau(X',s'_{k_1+1})\cdots \tau(X',s'_{k}) = \tau(X',t'_{\ell_1+1})\cdots \tau(X',t'_{\ell})$ hold in~$A^{\scriptscriptstyle +}$ (see the proof of Lemma~\ref{lemsecdecpodeslcmdesnu}).
From Theorem~\ref{theoremcentralisateurgeneralisationpicantin}, it follows that the two paths~$\tau(X',s'_{k_1+1})\cdots \tau(X',s'_{k})$ and~$\tau(X',t'_{\ell_1+1})\cdots \tau(X',t'_{\ell})$ are~$\equiv$-equivalent, using the relations of type~(\ref{relundelaprese}). Finally,~$\nu(Y'_1,s'_1)\cdots \nu(Y'_{k_1},s'_{k_1})$ and~$\nu(Z'_1,t'_1)\cdots \nu(Z'_{\ell_1},t'_{\ell_1})$ are~$\equiv$-equivalent by Lemma~\ref{lemtechnordset}.  Define~$X$ to be the set  made of the paths in~$\mathcal{C}(\mathcal{A})_{\cdot\to A^{\scriptscriptstyle +}_{X'}}$ whose images in~$\ConjplusA_{\cdot\to A^{\scriptscriptstyle +}_{X'}}$ right-divides~$\nu(Y'_1,s'_1)\cdots \nu(Y'_{k_1},s'_{k_1})$. We define a noetherian partial order on~$X$ by saying that for $\omega_1,\omega_2$ in $X$, one has~$\omega_1 \leq \omega_2$ if $\omega_1 = \omega_2$ in~$\mathcal{C}(\mathcal{A})_{\cdot\to A^{\scriptscriptstyle +}_{X'}}$, or the image of~$\omega_1$ in~$\ConjplusA$ strictly right-divides the image of~$\omega_2$. We define~$E$ as the set of pair of elements of~$X$ that have the same image in~$\ConjplusA_{\cdot\to A^{\scriptscriptstyle +}_{X'}}$. We fix for every element~$w$ in~$\ConjplusA_{\cdot\to A^{\scriptscriptstyle +}_{X'}}$ a representing path~$\hat{w}$ in~$\mathcal{C}(\mathcal{A})_{\cdot\to A^{\scriptscriptstyle +}_{X'}}$. We define $\phi : E\to E$ in the following way. Let $(\omega_1,\omega_2)$ belong to~$E$. Write $\omega_1 = \nu(U_1,u_1)\cdots \nu(U_r,u_r) $ and $\omega_2 = \nu(V_1,v_1)\cdots \nu(V_{r'},v_{r'})$ and assume they are distinct. Let $i$ be minimal such that $\nu(U_i,u_i)\neq \nu(V_i,v_i)$. Choose an arbitrary defining relation~$\nu(U_i,u_i)\,\omega_3 =  \nu(V_i,v_i)\,\omega_4$ of type~(\ref{deregalite}). There exists $w$ in~$\ConjplusA_{\cdot\to A^{\scriptscriptstyle +}_{X'}}$ such that~$\nu(U_1,u_1)\cdots\nu(U_i,u_i)\,\omega_3\,w$ is equal to the image of~$\omega_1$ in~$\ConjplusA_{\cdot\to A^{\scriptscriptstyle +}_{X'}}$. We set~$\phi_2(\omega_1,\omega_2) =\omega_3 \hat{w}$ and  $\phi_1(\omega_1,\omega_2) = \nu(U_{i+1},u_{i+1})\cdots \nu(U_r,u_r) $. By Lemma~\ref{lemtechnordset}, and with its notations, there exists~$n$ such that~$\psi_n(x,y) = (1,1,\cdots,1)$. The results follows easily.
\end{proof}
\begin{Rem} Given a Garside monoid with a right partial action on a set, Dehornoy shows in~\cite{Deh8} that under a technical hypothesis, a natural Garside category can be associated with this action. We note that the positive ribbon category~$\ConjplusA$ does not arised trivially in this way. Actually, we can not defined a right partial action of~$A^{\scriptscriptstyle +}$ on the set of its parabolic submonoids by setting for an atom~$s$ that $A_X^+ \cdot s= A_Y^+$  where~$Y$ is such that $\nu_X(s)$ belongs to $\ConjplusA_{A_X^+\to A_Y^+}$. For instance, consider the braid group~$B_4$ with the notations of Example~\ref{ExagpAT}; then $(A_{\{s_2\}}^+\cdot s_1)\cdot s_3 = A_{\{s_1\}}$ whereas $(A_{\{s_2\}}^+\cdot s_3)\cdot s_1 = A_{\{s_3\}}$. Indeed, we do not have the equality~$\nu_{\{s_2\}}\!(s_1)\,\nu_{\{s_1\}}\!(s_3) = \nu_{\{s_2\}}\!(s_3)\,\nu_{\{s_3\}}\!(s_1)$ but $$\nu_{\{s_2\}}\!(s_1)\,\nu_{\{s_1\}}\!(s_3)\,\nu_{\{s_1\}}\!(s_2) = \nu_{\{s_2\}}\!(s_3)\,\nu_{\{s_3\}}\!(s_1)\,\nu_{\{s_3\}}\!(s_2)$$ which is a relation of type~(\ref{deregalite}) in Theorem~\ref{theorpreseconjplusA}. 
 
\end{Rem}

\begin{Lem}\label{lemtechnordset} Let $(X,\leq)$ be a noetherian partially ordered set with an infimum element denoted by~$1$. Assume $E$ is a subset of $X\times X$ that contains~$\{(x,x)\mid x\in X\}$ and which is stabilized by the maps $(x,y)\mapsto (y,x)$. Assume $\phi : E\to E, (x,y)\mapsto(\phi_1(x,y),\phi_2(x,y))$ is such that $\phi(x,x) = (1,1)$ for $x$ in $E$, and $\phi_1(x,y)< x$  and $\phi_2(x,y)< x$ for $x\neq y$. Consider $\psi_n : E\to E^{2^n}$ defined inductively by $\psi_0(x,y) = (x,y)$ and~$\psi_n(x,y) = (\psi_{n-1}(\phi(x,y)),\psi_{n-1}(\phi(y,x)))$ for $n\geq 1$. For every $(x,y)$ in $E$, there exists~$n$ such that~$\psi_n(x,y) = (1,1,\cdots,1)$.
\end{Lem}
\begin{proof}
If it was not the case, then we could construct a strictly decreasing sequence in $X$, which is impossible by the noetherianity property.    
\end{proof}

In the sequel, we denote by~$\ConjplusAnu$ and~$\ConjAnu$ the subcategory and the subgroupoid of~$\ConjplusA$ and~$\ConjA$, respectively, generated by the~$\nu(X,s)$. 

\begin{Lem} (i) For every~$g$  which lies in~$\ConjplusA_{A_X^{\scriptscriptstyle +}\to\cdot}$, there exists a unique pair~$(g_1,g_2)$ such that~$g = g_1g_2$ with~$g_1$ in~$\QZ(A_X^{\scriptscriptstyle +})$ and~$g_2$ in~$\ConjplusAnu_{A_X^{\scriptscriptstyle +}\to\cdot}$.\\(ii) Let~$g,h$ belong to~$\ConjplusA_{A_X^{\scriptscriptstyle +}\to \cdot}$. Assume~$g = g_1g_2$ and~$h =h_1h_2$ with~$g_1,h_1$ in~$\QZ(A_X^{\scriptscriptstyle +})$ and~$g_2,h_2$ in~$\ConjplusAnu_{A_X^{\scriptscriptstyle +}\to\cdot}$. Then~$g\lor h$ is equal to~$(g_1\lor h_1)(g_2\lor h_2)$, and~$g\land h$is equal to~$(g_1\land h_1)(g_2\land h_2)$.        
\label{lemstructconjnu}\end{Lem}
\begin{proof} $(i)$ it has been proved when proving Theorem~\ref{theorpreseconjplusA}.\\
$(ii)$ It is clear that~$(g_1\lor h_1)(g_2\lor h_2)$ is a common multiple of~$g$ and~$h$. Conversely, any common multiple of~$g$ and~$h$ is a common multiple of~$(g_1\lor h_1)$ and~$(g_2\lor h_2)$ for the left-divisibility. But the lcm of the latter is~$(g_1\lor h_1)(g_2\lor h_2)$ by Lemma~\ref{mpolkiujyh} (the lcm of~$g_2$ and~$h_2$ lies in~$\ConjplusA_{A_X^{\scriptscriptstyle +}\to \cdot}$ by Lemma~\ref{lemsecdecpodeslcmdesnu}). Similarly,~$(g_1\land h_1)(g_2\land h_2)$ is a common left-divisor of~$g$ and~$h$. Now, let~$k = k_1k_2$ belong to~$\ConjplusA_{A_X^{\scriptscriptstyle +}\to\cdot}$ with~$k_1$ in~$\QZ(A_X^{\scriptscriptstyle +})$ and~$k_2$ in~$\ConjplusAnu_{A_X^{\scriptscriptstyle +}\to\cdot}$, and assume~$k$ left-divides both~$g$ and~$h$. Using that~$g$ is the lcm of~$g$ and~$k$ for the left-divisibility, we get that~$k_1$ and~$k_2$ left-divide~$g_1$ and~$g_2$, respectively. By the same argument,~$k_1$ and~$k_2$ also left-divide~$h_1$ and~$h_2$, respectively. We deduce that~$k_1$ and~$k_2$  left-divide~$(g_1\land h_1)$ and~$(g_2\land h_2)$, and finally that~$k$ left-divides~$(g_1\land h_1)(g_2\land h_2)$.  
 \end{proof}

\begin{Cor} Let~$(A^{\scriptscriptstyle +},\Delta)$ be a monoidal Garside structure such that~$A^{\scriptscriptstyle +}$ has  a~$\nu$-structure. Denote by~$S$ its atom set, and by~$A$ its group of fractions. The subgroupoid~$\ConjAnu$ of~$\ConjA$ generated by the~$\nu(X,s)$ is a standard parabolic subgroupoid of~$\ConjA$. Furthermore, Let~$\mathcal{A}^\nu$ denote the subgraph of~$\mathcal{A}$ with the same vertex set and whose edges are the~$\nu(X,s)$. Let~$R^\nu$ be the set of relations of type~{\rm (3)} that appear in Theorem~\ref{theorpreseconjplusA}. Then~$\langle\mathcal{A}^\nu,\mathcal{R}^\nu\rangle$ is a presentation of~$\ConjAnu$ as a groupoid.
\end{Cor}
\begin{proof}  We recall that the notion of a parabolic subcategory has been introduced in Definition~\ref{defparasucat}. We only need to prove that~$\ConjplusAnu$ is a parabolic subcategory of~$\Conjplus$ and that~$\langle\mathcal{A}^\nu,\mathcal{R}^\nu\rangle$ is a presentation of~$\ConjplusAnu$ as a category. Let~$A_X^{\scriptscriptstyle +}$ be a parabolic subcategory of~$A^{\scriptscriptstyle +}$ with~$\Delta_X$ as Garside element. The lattice~$\ConjplusAnu_{A_X^{\scriptscriptstyle +}\to\cdot}$ is a sublattice of~$\ConjplusA_{A_X^{\scriptscriptstyle +}\to\cdot}$ by Lemma~\ref{lemstructconjnu}. For similar reason,~$\ConjplusAnu_{\cdot\to A_X^{\scriptscriptstyle +}}$ is a sublattice of~$\ConjplusA_{\cdot\to A_X^{\scriptscriptstyle +}}$. By definition~$[\Delta](A_X^{\scriptscriptstyle +})$ is equal to~$\Delta$ and we can decompose~$\Delta$ as a product~$\Delta_X\nabla_X$. Clearly,~$\nabla_X$ lies in~$\ConjplusA_{A_X^{\scriptscriptstyle +}\to\cdot}$. As~$A_X^{\scriptscriptstyle +}$ is a parabolic submonoid, no element of~$X$ can left-divides~$\nabla_X$. This implies that~$\nabla_X$ lies in~$\ConjplusAnu_{A_X^{\scriptscriptstyle +}\to\cdot}$. We claim that the equalities $$[1,\nabla_X]_{\ConjplusA} = [1,\Delta]_{\ConjplusA}\bigcap\ {\ConjplusAnu}_{A_X^{\scriptscriptstyle +}\to\cdot} = \Delta\land{\ConjplusAnu}_{A_X^{\scriptscriptstyle +}\to\cdot}$$ hold. This proves that~$\ConjplusAnu$ is a parabolic subcategory of~$\Conjplus$. Indeed, this equalities are direct consequences of Lemma~\ref{lemstructconjnu}: assume~$g$ lie in~${\ConjplusAnu}_{A_X^{\scriptscriptstyle +}\to\cdot}$ and assume it belongs to~$[1,\Delta]_{\ConjplusA}$. Write~$g = g_1g_2$ with~$g_1$ in~$\QZ(A_X^{\scriptscriptstyle +})$ and~$g_2$ in~$\ConjplusAnu_{A_X^{\scriptscriptstyle +}\to\cdot}$. By Lemma~\ref{lemstructconjnu},~$g_1$ and~$g_2$ left-divide~$\Delta_X$ and~$\nabla_X$, respectively, and~$g_1 = 1$ if and only~$g$ left-divides~$\nabla_X$.
\end{proof}

We recall that in the case of an Artin-Tits group~$A$ with~$S$ as a generating set, every subset~$T$ of~$S$ generates a standard parabolic subgroup with~$T$ as atom set. In this particular case, we can precise the statement of the Theorem~\ref{theorpreseconjplusA}.
\begin{Prop} \label{antelastprop}
Let~$(A^{\scriptscriptstyle +},\Delta)$ be a monoidal Garside structure such that~$A^{\scriptscriptstyle +}$ is an Artin-Tits monoid of spherical type. Denote by~$S$ its atom set and by~$A$ its associated Artin-Tits group. Assume that~$X$ is a non empty subset of~$S$ and let~$s,t$ lie in~$S-X$ and be distinct. Then~$\nu(X,s)\lor\nu(X,t)$ has exactly two representing path in the free category~$\mathcal{C}(\mathcal{A})$. One begins with~$\nu(X,s)$ and the other with~$\nu(X,t)$. 
\end{Prop}

\begin{proof} One has~$\nu(X,s) = \Delta_X^{-1}\Delta_{X\cup\{s\}}$ and~$\nu(X,t) = \Delta_X^{-1}\Delta_{X\cup\{t\}}$ ({\it cf.} Example~\ref{exemplnufunctspherATGP}). Therefore,~$\nu(X,s)\lor\nu(X,t)$ is equal to~$\Delta_X^{-1}\Delta_{X\cup\{s,t\}}$ and, in particular,~$\Delta_X(\nu(X,s)\lor\nu(X,t))$ is equal to~$\Delta_{X\cup\{s,t\}}$, which is a simple element of~$A^{\scriptscriptstyle +}$ ({\it c.f.} Definition~\ref{defisimplelem}). Consider in~$\mathcal{C}(\mathcal{A})$ a representing path~$\nu(Y_0,s_0)\nu(Y_1,s_1)\cdots \nu(Y_k,s_k)$  of~$\nu(X,s)\lor\nu(X,t)$ such that~$Y_0 = X$ and~$s_0 = s$. Given~$Y_i$ and~$s_i$, the set~$Y_{i+1}$ is uniquely defined by the equality~$Y_i\nu(Y_i,s_i) = \nu(Y_i,s_i)Y_{i+1}$. We can write $$\displaylines{\Delta_X(\nu(X,s)\lor\nu(X,t)) = \Delta_{Y_0\cup\{s_0\}}\nu(Y_1,s_1)\cdots \nu(Y_k,s_k) =\hfill\cr\hfill \nu(Y_0,s_0)\Delta_{Y_1\cup\{s_1\}}\nu(Y_2,s_2)\cdots \nu(Y_k,s_k) =\nu(Y_0,s_0)\cdots \nu(Y_{k-1},s_{k-1}) \Delta_{Y_k\cup\{s_k\}}.}$$ Now it is well-known that the simple elements of an Artin-Tits monoid are \emph{square free}: no representing word of a simple element has a square of an element of~$S$ as a subfactor. This implies that for every~$1\leq i\leq k$, the element~$s_i$ is uniquely defined by the two conditions~$s_i \in S\cup\{s,t\}$ and~$s_i\not\in\{s_{i-1}\}\cup Y_{i-1}$. Thus, there is a unique representing path of~$\nu(X,s)\lor\nu(X,t)$ in~$\mathcal{C}(\mathcal{A})$ that starts with~$\nu(X,s)$. Similarly, there is a unique representing path of~$\nu(X,s)\lor\nu(X,t)$ in~$\mathcal{C}(\mathcal{A})$ that starts with~$\nu(X,t)$.
\end{proof}

\subsection{Weak Garside group} 
In~\cite{Bess}, Bessis defined a \emph{weak Garside group} to be a group that is isomorphic to some vertex group~$\mathcal{C}_{x\to x}$ of a Garside groupoid~$\mathcal{C}$. Clearly a Garside group is a weak Garside group, but there is no obvious general reason for a weak Garside group to be a Garside group~\cite{Bess}. In this section, we consider a Garside group~$A$ with a~$\nu$-function and the subgroup~$G$ of some vertex group~$\ConjA_{A_X\to A_X}$ that is generated by the atoms of~$\ConjA$ that belong to~$\ConjA_{A_X\to A_X}$. Under some assumptions that hold in the case of an Artin-Tits group of spherical type, we prove ({\it cf.}~Theorem~\ref{THgarssubg}) that~$G$ is a Garside group. This section is inspired by~\cite[Theorem B]{BrHo}.

We still assume that~$(A^{\scriptscriptstyle +},\Delta)$ is monoidal Garside structure such that~$A^{\scriptscriptstyle +}$ has a~$\nu$-structure. We still denote by~$S$ and~$A$ the atom set and the group of fractions of~$A^{\scriptscriptstyle +}$, respectively. For all the section, we fix a standard parabolic subgroup~$A_X$ of~$A$. We set $${\Sh}(X) = \{\nu_X(s) \mid s\in S\ ;\  \nu_X(s)\in  {\ConjA}_{A_X\to A_X}\}.$$ Following \cite{BrHo}, we call the elements of~${\Sh}(X)$ the \emph{shakers} of~$A_X$. 
 It is immediate that the set of shakers~${\Sh}(X)$ is equal to~$\{\tilde{\nu}_X(s) \mid s\in S\ ;\  \tilde{\nu}_X(s)\in  {\ConjA}_{A_X\to A_X}\}$. There is two kinds of shakers, namely the~$\tau(X,s)$ and the~$\nu(X,s)$; in the sequel, we denote by~${\Sh}^\nu(X)$ the set $\{\nu_X(s) \in {\Sh}(X)\mid \nu_X(s) = \nu(X,s)\}$, which is equal to the set~$\{\tilde{\nu}_X(t) \in {\Sh}(X)\mid \tilde{\nu}_X(t) = \tilde{\nu}(X,t)\}$. Furthermore, we denote by~$\Lsh(X)$ and by~$\Rsh(X)$ the sets $\nu_X^{-1}({\Sh}(X))$ and~$\tilde{\nu}_X^{-1}({\Sh}(X))$ respectively. There is no reason to expect that~$\Rsh(X) = \Lsh(X)$ in general.  Similarly, we set~$\Lsh^\nu(X) = \nu_X^{-1}({\Sh}^\nu(X))$ and~$\Rsh^\nu(X) = \tilde{\nu}_X^{-1}({\Sh}^\nu(X))$. 

Finally, in the sequel we extend the notation previously used for parabolic submonoids and parabolic subgroups: for every subset~$X$ of~$A^{\scriptscriptstyle +}$, we denote by~$A^{\scriptscriptstyle +}_{X}$ and~$A_{X}$ the submonoid of~$A^{\scriptscriptstyle +}$ and the subgroup of~$A$, respectively, generated by~$X$. Below, we consider~$A^{\scriptscriptstyle +}_{\Sh(X)}$ and~$A^{\scriptscriptstyle +}_{{\Sh}^\nu(X)}$, which are submonoids of~$\ConjplusA_{A_X^{\scriptscriptstyle +}\to A_X^{\scriptscriptstyle +}}$. 
\begin{Lem} \label{LemTHgarsubg} Assume~$A^{\scriptscriptstyle +}_{\Lsh(X)}$ is a parabolic submonoid of~$A^{\scriptscriptstyle +}$, and denote by~$\Delta_{\Lsh(X)}$ its Garside element. Then,\\(i) the elements of~${\Sh}(X)$ belong to~$A^{\scriptscriptstyle +}_{\Lsh(X)}$.
\\ (ii) One has~$\ConjplusA_{A_X^{\scriptscriptstyle +}\to\cdot} \cap\ A^{\scriptscriptstyle +}_{\Lsh(X)} =  A^{\scriptscriptstyle +}_{\Sh(X)}$, and every atom of~$\ConjplusA$ that appears in any decomposition of~$\Delta_{\Lsh(X)}$ as a product of atoms  in~$\ConjplusA_{A_X^{\scriptscriptstyle +}\to\cdot}$   belongs to~${\Sh}(X)$.\\
(iii) One has~$\ConjplusAnu_{A_X^{\scriptscriptstyle +}\to\cdot} \cap\ A^{\scriptscriptstyle +}_{\Lsh^\nu(X)} =  A^{\scriptscriptstyle +}_{\Sh(X)^\nu}$. Write~$\Delta_{\Lsh(X)} = \Delta_X\nabla_X$. Then~$\nabla_X$ belongs to~$A^{\scriptscriptstyle +}_{\Sh(X)^\nu}$, and  every atom of~$\ConjplusAnu$ that appears in any decomposition of~$\nabla(X)$ in~$\ConjplusAnu_{A_X^{\scriptscriptstyle +}\to\cdot}$ as a product of atoms belongs to~${\Sh}^\nu(X)$.
\end{Lem}
\begin{proof}(i) All the elements of~$\Lsh(X)$ left-divide~$\Delta_{\Lsh(X)}$, which lies in~$\ConjplusA_{A_X^{\scriptscriptstyle +}\to\cdot}$. Then, by the defining properties of the~$\nu$-structure, it follows that the elements of~$\Sh(X)$ left divides~$\Delta_{\Lsh(X)}$. But~$A^{\scriptscriptstyle +}_{\Lsh(X)}$ is a parabolic submonoid, therefore all the elements of $\Sh(X)$ belong to~$A^{\scriptscriptstyle +}_{\Lsh(X)}$.\\
(ii) The set~$A^{\scriptscriptstyle +}_{\Sh(X)}$ is included in~$\ConjplusA_{A_X^{\scriptscriptstyle +}\to\cdot} \cap\ A^{\scriptscriptstyle +}_{\Lsh(X)}$ by~$(i)$. Conversely, consider an element~$g$ of~$A^{\scriptscriptstyle +}$ which lies in~$\ConjplusA_{A_X^{\scriptscriptstyle +}\to\cdot} \cap\ A^{\scriptscriptstyle +}_{\Lsh(X)}$ and choose an arbitrary decomposition~$\nu_{Y_0}(s_0)\nu_{Y_1}(s_1)\cdots \nu_{Y_k}(s_k)$ of~$g$ in~$\ConjplusA_{A_X^{\scriptscriptstyle +}\to\cdot}$ (where~$Y_0 = X$). By definition of a~$\nu$-structure, we can assume without restriction that each~$s_i$ left-divides~$\nu_{Y_i}(s_i)$; otherwise, if we consider~$s'_i$ that left-divides~$\nu_{Y_i}(s_i)$, then we have~$\nu_{Y_i}(s'_i)  = \nu_{Y_i}(s_i)$. The assumption that~$A^{\scriptscriptstyle +}_{\Lsh(X)}$ is a parabolic submonoid  implies that all the~$s_i$ belong to~$A^{\scriptscriptstyle +}_{\Lsh(X)}$. Therefore,  we get step-by-step that all the~$Y_i$ are equal to~$X$. Then~$g$ belongs to~$A^{\scriptscriptstyle +}_{\Sh(X)}$. If we applied the above argument to~$g =\Delta_{\Lsh(X)}$ we get the second part of~$(ii)$.\\
(iii) The element~$\Delta_X$ belongs to~$\QZ(A_X^{\scriptscriptstyle +})$ and no element of~$X$ can left-divides~$\nabla_X$. Indeed, if an~$s$ in~$S$ left-divides~$\nabla_X$, then~$\Delta_Xs$ is a simple element of~$A^{\scriptscriptstyle +}$ and therefore cannot lie in~$A_X^{\scriptscriptstyle +}$. Then,~$\nabla_X$ lies in~$A^{\scriptscriptstyle +}_{\Sh(X)^\nu}$ ({\it cf.} Lemma~\ref{lemstructconjnu}$(i)$) and the atoms of~$\ConjplusAnu$ that appears in any decomposition of~$\nabla(X)$ belongs to~${\Sh}^\nu(X)$.
\end{proof}
\begin{The} \label{THgarssubg} Let~$(A^{\scriptscriptstyle +},\Delta)$ be a monoidal Garside structure such that~$A^{\scriptscriptstyle +}$ has a $\nu$-structure. Denote by~$S$ its atom set and by~$A$ its associated Garside group. Let $A^{\scriptscriptstyle +}_X$ be a parabolic submonoid of~$A^{\scriptscriptstyle +}$ such that~$\Lsh(X) = \Rsh(X)$. If the monoid~$A^{\scriptscriptstyle +}_{\Lsh(X)}$ is a parabolic submonoid of~$A^{\scriptscriptstyle +}$ whose Garside element is denoted by~$\Delta_{\Lsh(X)}$, then\\
(i) the subgroup~$A_{{\Sh}(X)}$ is a standard parabolic subgroup of the groupoid~$\ConjA$ with~$A_X$ as unique object. In particular,~$A_{{\Sh}(X)}$ is a Garside group. The Garside element of~$A_{\Sh(X)}^{\scriptscriptstyle +}$ is~$\Delta_{\Lsh(X)}$ and its atom set is~${\Sh}(X)$.\\
(ii) Write~$\Delta_{\Lsh(X)} = \Delta_X\nabla_X$. The subgroup~$A_{{\Sh}^\nu(X)}$ of~$A$ is a (standard) parabolic subgroup of~$A_{{\Sh}(X)}$ with~$\nabla_X$ for Garside element, and $$A_{{\Sh}(X)} = \QZ(A_{\Lsh(X)})\rtimes A_{{\Sh}^\nu(X)}.$$
\end{The}

\begin{proof} $(i)$ The element~$\Delta_{\Lsh(X)}$ belongs to~$A^{\scriptscriptstyle +}_{{\Sh}(X)}$ and we have   $$\displaylines{[1_{A_X^{\scriptscriptstyle +}},\Delta_{\Lsh(X)}]_{\ConjplusA} \subseteq [1_{A_X^{\scriptscriptstyle +}},\Delta]_{\ConjplusA} \cap A^{\scriptscriptstyle +}_{{\Sh}(X)} \subseteq \left(\Delta\wedge A^{\scriptscriptstyle +}_{\Lsh(X)}\right)\cap {\ConjplusA}_{A_X^{\scriptscriptstyle +}\to\cdot} \hfill\cr\hfill \subseteq \left(\Delta_{{\Lsh}(X)}\wedge A^{\scriptscriptstyle +}_{\Lsh(X)}\right)\cap {\ConjplusA}_{A_X^{\scriptscriptstyle +}\to\cdot} \subseteq [1_{A_X^{\scriptscriptstyle +}},\Delta_{\Lsh(X)}]_{\ConjplusA}.}$$ 
The first inclusion follows from Lemma~\ref{LemTHgarsubg}$(ii)$, and  the third one from the fact that $A^{\scriptscriptstyle +}_X$ be a parabolic submonoid of~$A^{\scriptscriptstyle +}$. By Lemma~\ref{LemTHgarsubg}$(ii)$, we also get that~$A^{\scriptscriptstyle +}_{{\Sh}(X)}$ is a sublattice of~$\ConjplusA_{A_X^{\scriptscriptstyle +}\to\cdot}$ for the left-divisibility: let~$g$ and~$h$ belong to~$A^{\scriptscriptstyle +}_{{\Sh}(X)}$, then their gcd and their lcm for the left-divisibility lie in~$A_{{\Lsh}(X)}^{\scriptscriptstyle +}$ because the latter is a parabolic submonoid of~$A^{\scriptscriptstyle +}$. But, they also lie in~$\ConjplusA_{A_X^{\scriptscriptstyle +}\to\cdot}$.  Then, by Lemma~\ref{LemTHgarsubg}$(ii)$, they lie in~$A^{\scriptscriptstyle +}_{{\Sh}(X)}$, and  they are common divisor and common multiple, respectively, of~$g$ and~$h$ in~$A_{{\Lsh}(X)}^{\scriptscriptstyle +}$. Since we assume~$\Lsh(X) = \Rsh(X)$, by symmetry,~$A^{\scriptscriptstyle +}_{{\Sh}(X)}$ is a sublattice of~$\ConjplusA_{\cdot\to A_X^{\scriptscriptstyle +}}$ for the right-divisibility.\\$
(ii)$ We have to prove that~$A^{\scriptscriptstyle +}_{{\Sh}^\nu(X)}$ is a  parabolic submonoid of~$A^{\scriptscriptstyle +}_{{\Sh}(X)}$, with~$\nabla_X$ as Garside element. This is a consequence of Lemma~\ref{LemTHgarsubg}$(iii)$ and~\ref{lemstructconjnu} with a proof similar to~$(ii)$. It is clear that~$A^{\scriptscriptstyle +}_{{\Sh}^\nu(X)}$ normalizes~$\QZ^{\scriptscriptstyle +}(A_{\Lsh(X)})$ and both subgroups generate~$A_{{\Sh}(X)}$. Finally, we have a semi-direct product because the submonoids~$\QZ^{\scriptscriptstyle +}(A_{\Lsh(X)})$ and $A^{\scriptscriptstyle +}_{{\Sh}^\nu(X)}$ are parabolic with a trivial intersection. Indeed, the left greedy normal of every element in the intersection of the parabolic subgroups~$\QZ(A_{\Lsh(X)})$ and $A_{{\Sh}^\nu(X)}$ has to be trivial. \end{proof}

In order to state the next result, we recall that an Artin-Tits group of spherical type is said to be \emph{indecomposable} when it is not the direct product of two of its proper standard parabolic subgroups.
\begin{Prop}\label{lastprop} Let~$(A^{\scriptscriptstyle +},\Delta)$ be a monoidal Garside structure such that~$A^{\scriptscriptstyle +}$ is an Artin-Tit monoid of spherical type, with~$S$ as set of atoms and~$A$ as associated Artin-Tits group. Assume that~$X$ is a subset of~$S$. The subgroup~$A_{{\Sh}(X)}$ is an Artin-Tits group of spherical type with~${\Sh}(X)$ as associated atom set. In particular, if~$s,t$ are distinct in~$S-X$, the associated Relation~{\rm (\ref{deregalite})} of the presentation of~$\ConjplusA$ stated in Theorem~\ref{theorpreseconjplusA} is $$\underbrace{\nu(X,t)\nu(X,s)\nu(X,t)\cdots}_{M_{X,s,t}\ terms} = \underbrace{\nu(X,s)\nu(X,t)\nu(X,s)\cdots}_{M_{X,s,t}\ terms}$$
for some positive integer~$M_{X,s,t}$. When~$A_S$ is indecomposable with~$\# S \geq 3$, the subgroups~$A_{{\Sh}(X)}$ of~$A$ where~$X\subset S$ are classified by the \emph{tables of relations R2} of \cite[Appendix 5]{BrHo}.  
\end{Prop}

\begin{proof} By Theorem~\ref{THgarssubg}, we know that~$A_{{\Sh}(X)}$ is a Garside group with~${\Sh}(X)$ as associated atom set. Now, consider the decomposition~$\nu(X,s)\nu(Y_1,s_1)\cdots \nu(Y_k,s_k)$ of~$\nu(X,s)\lor\nu(X,t)$ that starts with~$\nu(X,s)$ in~$\mathcal{C}(\mathcal{A})$ ({\it cf.} Proposition~\ref{antelastprop}). By assumption,~$Y_1 = X$. By the argument of the proof of Proposition~\ref{antelastprop}, we get~$s_1 = t$, then~$Y_2 = X$, and step-by-step that~$\nu(X,s)\lor\nu(X,t) = \underbrace{\nu(X,s)\nu(X,t)\nu(X,s)\cdots}_{k\textrm{ terms}}$. Similarly~$\nu(X,s)\lor\nu(X,t)  = \underbrace{\nu(X,t)\nu(X,s)\nu(X,t)\cdots}_{\ell\textrm{ terms}}$.
Finally,~$k = \ell$ because otherwise we could decompose~$\nu(X,s)\lor\nu(X,t)~$ as a product~$w_1(\nu(X,s)\lor\nu(X,t))w_2$ with~$(w_1,w_2)\neq (1,1)$, which is impossible. Then~$A_{{\Sh}(X)}$ is an Artin-Tits group of spherical type with~${\Sh}(X)$ as associated atom set. The proof of the remaining is a case by case computation, which is the same as in \cite{BrHo}. \end{proof}
\subsection{Semi-direct product} 
Theorem~\ref{THgarssubg} of the previous section claims that the group~$A_{{\Sh}(X)}$ is a semi-direct product of two of its subgroups. In order to established that result, we use Lemma~\ref{lemstructconjnu}~$(i)$. If we were dealing with groups instead of categories in the latter, we would have use the notion of a  group semi-direct product in order to state it. The notion of a \emph{semi-direct product of categories}  is not so easy to defined, even though one can find a definition in~\cite{Ehr}, or more recently in~\cite{Ste}. The notion of a \emph{Zappa-Sz\'ep product}, as covered in~\cite{Bri}, appears as more natural.
\begin{Def}[Zappa-Sz\'ep product]\cite[Lemma~3.2 and~3.9]{Bri} Let~$\mathcal{C}$ be a category and~$\mathcal{D},\mathcal{D}'$ be two subcategories of~$\mathcal{C}$. We say that~$\mathcal{C}$ is an (\emph{internal}) \emph{Zappa-Sz\'ep product} of~$\mathcal{D}$ and~$\mathcal{D}'$ if every morphism~$g$ of~$\mathcal{C}$ can be uniquely decomposed as a product~$g_1g_2$ where~$g_1$ is a morphism of~$\mathcal{D}$ and~$g_2$ is a morphism of~$\mathcal{D}'$. In this case, we write~$\mathcal{C} = \mathcal{D}\bowtie \mathcal{D}'$. 
\end{Def}

Then, Lemma~\ref{lemstructconjnu}$(i)$ claims that $\ConjplusA = \mathcal{QZ}^{\scriptscriptstyle +}(A) \bowtie \ConjplusAnu$ where $\mathcal{QZ}^{\scriptscriptstyle +}(A)$ be the totally disconnected category whose objects are the parabolic submonoids of $A^{\scriptscriptstyle +}$ and whose vertex monoid at $A^{\scriptscriptstyle +}_X$ is $\QZ^{\scriptscriptstyle +}(A^{\scriptscriptstyle +}_X)$. Clearly, in our context, we have more than a Zappa-Sz\'ep product. Let us introduce the following definition:

\begin{Def}[Semi-direct product] Let~$\mathcal{C}$ be a category and~$\mathcal{D},\mathcal{D}'$ be two subcategories of~$\mathcal{C}$ such that~$\mathcal{C} = \mathcal{D}\bowtie \mathcal{D}'$. Assume~$\mathcal{D}$ is equal to a totally disconnected category that is the union of its vertex monoids~${\mathcal{D}}_{x\to x}$.  We say that~$\mathcal{C}$ is a \emph{semi-direct product} of~$\mathcal{D}$ and~$\mathcal{D}'$ when for every $g$ in ${\mathcal{D}'}_{x\to y}$ one has~${\mathcal{D}}_{x\to x}g = g{\mathcal{D}}_{y\to y}$. In this case, we write $\mathcal{C} = \mathcal{D}\rtimes \mathcal{D}'$.
\end{Def}
So, for instance one has $\ConjplusA = \mathcal{QZ}^{\scriptscriptstyle +}(A) \rtimes \ConjplusAnu$. This definition extends the standard definition of a semi-direct product of monoids, and is a special case of the definition considered in~\cite{Ste}. We finish this article with the following result which generalizes~\cite[Prop.~4]{God_jal} and provides a nice description of the connection between the standard parabolic subgroups of a Garside group.  
\begin{The} \label{theodescnormal}Let~$(A^{\scriptscriptstyle +},\Delta)$ be a monoidal Garside structure such that~$A^{\scriptscriptstyle +}$ has a~$\nu$-structure. Denote by~$A$ the group of fractions of $A^{\scriptscriptstyle +}$. \\(i) Let $\mathcal{QZ}(A)$ be the totally disconnected groupoid whose objects are the parabolic subgroups of $A$ and whose vertex group at $A_X$ is $\QZ(A_X)$. Then, $$\ConjA = \mathcal{QZ}(A)\rtimes \ConjAnu.$$ 
(ii) Let~$Conj(A)$ be the  groupoid whose objects are the parabolic subgroups of $A$, whose morphism set~$Conj(A)_{A_X\to A_Y}$ from~$A_X$ to~$A_Y$ is $\{g\in A\mid g^{-1}A_Xg = A_Y\}$, and whose composition is the product in~$A$. Let $\mathcal{P}(A)$ be the totally disconnected groupoid whose objects are the parabolic subgroups of $A$ and whose vertex group at $A_X$ is $A_X$.  Then, $$Conj(A) = \mathcal{P}(A)\rtimes \ConjAnu.$$ 
\end{The}
Note that~$Conj(A)$ is the full subgroupoid of~$\NA$ (see Definition~\ref{defngc}) whose objects are the parabolic subgroups of $A$. Its vertex groups are normalizers of standard parabolic subgroups. This result holds in particular when $A$ is an Artin-Tits group of spherical type. 
\begin{proof}
$(i)$ The unique not obvious fact is that $\ConjA$ is equal to the Zappa-Sz\'ep product~$\mathcal{QZ}(A) \bowtie \ConjAnu$. This is a consequence of the Zappa-Sz\'ep product decomposition of~$\ConjplusA$, since $\ConjA$, $\mathcal{QZ}(A)$ and $\ConjAnu$ are the groupoids of fractions of~the Garside categories~$\ConjplusA$, $\mathcal{QZ}^{\scriptscriptstyle +}(A^{\scriptscriptstyle +})$ and $\ConjplusAnu$, respectively (where $\mathcal{QZ}^{\scriptscriptstyle +}(A^{\scriptscriptstyle +})$ is the totally disconnected category whose objects are the parabolic submonoids of $A^{\scriptscriptstyle +}$ and whose vertex monoid at $A^{\scriptscriptstyle +}_X$ is $\QZ(A^{\scriptscriptstyle +}_X)$).  \\
$(ii)$ Again the unique not obvious fact is that Conj(A) = $\mathcal{P}(A)\bowtie \ConjAnu$. The unicity of the decomposition follows easily from Lemma~\ref{lemstructconjnu}$(ii)$. Let $g$ lie in~$A$ such that $g^{-1}A_Xg = A_Y$. We show there exist $a$ in $A_X$ and $r$ in $\ConjAnu_{A_X\to A_Y}$ such that $g = ar$. There exists $n\in\mathbb{N}$ and $g_1$ in $A^{\scriptscriptstyle +}$ such that $g = g_1\Delta^{-n}$. Furthermore, there exists a parabolic submonoid~$A^{\scriptscriptstyle +}_{Z}$ of $A^{\scriptscriptstyle +}$ such that $\Delta^n$ belongs to $\ConjplusA_{A^{\scriptscriptstyle +}_Y\to A^{\scriptscriptstyle +}_Z}$. In particular, we have $g_1^{-1}A_Xg_1 = A_Z$. We can write $g_1 = a_1r_1b_1$ such that $a_1$ belongs to $A_X^{\scriptscriptstyle +}$, $b_1$ belongs to $A_Z^{\scriptscriptstyle +}$, no elements of $X$ left-divides~$r_1$ and no element of $Z$ right-divides $r_1$. Now, let $x$ belongs to $X$ and $c^{-1}z$ be the left greedy normal form of $r_1^{-1}xr_1$ in $A$. There exists $d$ in $A^{\scriptscriptstyle +}$ such that $xr_1 = dz$ and $r_1 = dc$. Since~$r_1^{-1}xr_1$ belongs to $A_Z$, the elements~$c$ and $z$ belongs to $A_Z^{\scriptscriptstyle +}$. By the assumption on $r_1$, this implies~$c = 1$ and $r_1^{-1}xr_1 = z$. Therefore $r_1A_X^{\scriptscriptstyle +} r_1^{-1}$ is included in $A_Z^{\scriptscriptstyle +}$ and by symmetry, $r_1A_X^{\scriptscriptstyle +} r_1^{-1} = A_Z^{\scriptscriptstyle +}$. Therefore, $r_1$ lies in $\ConjplusA_{A_X^{\scriptscriptstyle +}\to A_Z^{\scriptscriptstyle +}}$ and we have $g = a_1r_1b_1\Delta^{-n} = a_1a_2r_1\Delta^{-n}$ with $a_1a_2$ in $A_X^{\scriptscriptstyle +}$ and $r_1\Delta^{-n}$ in $\ConjA_{A_X\to A_Y}$. Using~$(i)$, we are done.
\end{proof}


\end{document}